\documentclass
{amsart}

\usepackage{color}
\usepackage{amsmath}
\usepackage{amssymb}
\usepackage{amsfonts}
\usepackage[abbrev]{amsrefs}
\usepackage{stmaryrd}
\usepackage{graphicx}
\usepackage{bm}

\definecolor{red}{rgb}{1.0,0.0,0.0}

\definecolor{blu}{rgb}{0.0,0.0,1.0}

\definecolor{gre}{rgb}{0.03,0.50,0.03}

\newtheorem{theorem}{Theorem}[section]
\newtheorem{proposition}[theorem]{Proposition}
\newtheorem{lemma}[theorem]{Lemma}
\newtheorem{cor}[theorem]{Corollary}

\theoremstyle{definition}
\newtheorem{definition}[theorem]{Definition}

\theoremstyle{definition}
\newtheorem{remark}[theorem]{Remark}

\numberwithin{equation}{section}

\newcommand{\bfone}{{\mathbf{1}}}
\newcommand{\eps}{{\varepsilon}}

\newcommand{\abs}[1]{\left|{#1}\right|}
\newcommand{\norm}[1]{\lVert{#1}\rVert}

\newcommand{\N}{{\mathbb{N}}}

\newcommand{\R}{{\mathbb{R}}}

\begin{document}
\title[Viability and l.s.c.~solutions]{Viability for locally monotone evolution inclusions and lower semicontinuous
solutions of\\ Hamilton--Jacobi--Bellman equations in\\ infinite dimensions}

\author{Jichao Jiang}
\author{Christian Keller}
\address
{Department of Mathematics, 
University of  Central Florida,
Orlando, FL 32816, United States}
\thanks{This research  was
 supported in part  by NSF-grant DMS-2106077.}
\email{jichao.jiang@ucf.edu}
\email{christian.keller@ucf.edu}

\date{November 26, 2024}


\begin{abstract}
We establish necessary and sufficient conditions for viability of evolution inclusions with locally monotone operators
in the sense of  Liu and R\"ockner [\textit{J.~Funct.~Anal.}, 259 (2010), pp.~2902-2922].
This allows us to prove wellposedness of lower semicontinuous solutions of Hamilton--Jacobi--Bellman equations
associated to the optimal control of evolution inclusions.
Thereby, we generalize results in  Bayraktar and Keller [\textit{J.~Funct.~Anal.}, 275 (2018), pp.~2096-2161] 
on Hamilton--Jacobi equations in infinite dimensions with monotone operators
in several ways. First, we permit locally monotone operators. This extends the applicability of our theory to a wider class of equations
such as Burgers' equations, reaction-diffusion equations, and 2D Navier--Stokes equations.
Second,  {\color{black} our results apply to optimal control problems
with state constraints.} Third, we have uniqueness of viscosity solutions.
Our results on viability and lower semicontinuous solutions are new even in the case of monotone operators.
\end{abstract}
\maketitle
\pagestyle{plain}
\tableofcontents

\section{Introduction}
We study viability problems involving evolution inclusions of the form
\begin{align}\label{E:1st}
x^\prime(t)+A(t,x(t))\in F(t,x(t))\quad\text{a.e.~on $(0,T)$,}
\end{align}
where $A$ is a locally monotone operator on a Gelfand triple $V\subset H\subset V^\ast$.
Evolution equations with locally monotone operators were introduced in \cite{Liu11Nonlin,LiuRoeckner10JFA}.
They generalize  equations with 
monotone operators. Many important equations such as Navier--Stokes equations, reaction-diffusion
equations, Burgers' equations, etc.~can be written as abstract evolution
equations involving locally monotone operators.

The main results of this work are necessary and sufficient criteria for some set $K$ to be \emph{viable} for \eqref{E:1st},
i.e., criteria that a solution $x$ of $\eqref{E:1st}$ satisfies $x(t)\in K$ for all $t\in [0,T]$ provided it starts in $K$.
Previous works on viability for evolution inclusions with monotone operators on Gelfand triples such as 
\cite{AP89Yokohama,HP2,P92Yokohama,Shi88Nonlin}
seem to only provide sufficient criteria.
Note that there are results with  strength  similar than ours (necessary and sufficient conditions for viability)
for evolution inclusions with operators that generate semigroups (see, e.g., \cite{Carja07book,Carja09TAMS}
and the references therein).

Thanks to our  mentioned main results on viability for \eqref{E:1st}, we obtain applications
for Hamilton--Jacobi equations in infinite dimensions. More precisely, we
prove existence and uniqueness for appropriate nonsmooth solutions of the following
path-dependent Hamilton--Jacobi--Bellman equation:
\begin{equation}\label{E:2nd}
\begin{split}
&\partial_t u(t,x)+\langle A(t,x(t)),\partial_x u(t,x)\rangle 
+\inf_{f\in F(t,x(t))} (f,\partial_x u(t,x))=0,\\ &\qquad\qquad (t,x)\in[0,T)\times C([0,T],H).
\end{split}
\end{equation}
This is a generalization of results in \cite{BK18JFA}, where the operator $A$ was required to be
monotone. Moreover, our solutions of \eqref{E:2nd} need only be lower semicontinuous.
This allows us to  characterize value functions for optimal control problems with  state constraints
related to \eqref{E:1st}  
as unique solutions to \eqref{E:2nd} in a similar way as it has been done in the finite-dimensional case
in \cite{Frankowska93SICON}.

\section{Setting}
Let $(V,H,V^\ast)$ be a Gelfand triple, i.e., we have $V\subset H\subset V^\ast$, where 
$V$ is a reflexive and separable Banach space
that is continuously and densily embedded into $H$, which is a Hilbert space, and $V^\ast$ is the dual space
of $V$.
Also assume that the mentioned embedding is compact.
We write  $\norm{\cdot}:=\norm{\cdot}_V$,
$\abs{\cdot}:=\norm{\cdot}_H$, and $\norm{\cdot}_\ast:=\norm{\cdot}_{V^\ast}$ for the corresponding norms. 
We also use $\abs{\cdot}$ for norms on Euclidean spaces.
We write $(\cdot,\cdot)$ for the inner product in $H$ and $\langle \cdot,\cdot\rangle$ for the duality pairing
between $V$ and $V^\ast$. Moreover, assume that $(h,v)=\langle h,v\rangle$ for each $h\in H$ and $v\in V$.
For further details, see section~2 in \cite{BK18JFA},
where the same setting is used. 

\subsection{Notation}

Balls in $H$ centered at the origin are denoted by
\begin{align*}
B(0,r):=\{x\in H:\abs{x}\le r\},\,r>0.
\end{align*}

Given two subsets $E_1$, $E_2$ of  a group, put
\begin{align*}
E_1+E_2:=\{e_1+e_2:\, \text{$e_1\in E_1$ and $e_2\in E_2$}\}.
\end{align*}

Given a function $f$ from some set $E$ to $\R\cup\{+\infty\}$, put
\begin{align*}
\mathrm{dom}\,f&:=\{x\in E:\,f(x)<+\infty\}\qquad\text{(\emph{effective domain of $f$})},\\
\mathrm{epi}\,f&:=\{(x,y)\in E\times\R:\, y\ge f(x)\}\qquad\text{(\emph{epigraph of $f$})}.
\end{align*}

We  borrow the  following notation from \cite{Carja12SICON}:
\begin{align*}
E_{L^2}:=\{f\in L^2(0,T;H):\,f(t)\in E\text{ a.e.~on $(0,T)$}\},\quad E\subset H,
\end{align*}

We write l.s.c.~for lower semicontinuous and u.s.c.~for upper semicontinuous.

\subsection{Path spaces}
We frequently use the space $C([0,T],H)$, which
we consider to be equipped with the supremum norm $\norm{\cdot}_\infty$.
Subsets of $[0,T]\times C([0,T],H)$ are considered to be equipped with the
pseudo-metric $\mathbf{d}_\infty$ defined 
by 
\begin{align*}
\mathbf{d}_\infty(t_1,x_1;t_2,x_2):=\abs{t_2-t_1}+\norm{x_1(\cdot\wedge t_1)-x_2(\cdot\wedge t_2)}_\infty.
\end{align*}
Here, $\wedge$ means minimum, i.e., $a\wedge b:=\min\{a,b\}$. Continuity and semicontinuity of
functions defined on subsets of $[0,T]\times C([0,T],H)$ are always to be understood with respect to $\mathbf{d}_\infty$.

\begin{remark}
A semicontinuous function $u:[0,T]\times C([0,T])\supset E\to\R\cup\{+\infty\}$ is \emph{nonanticipating},
i.e., whenever $x_1=x_2$ on $[0,t]$, then $u(t,x_1)=u(t,x_2)$. This follows immediately from
the definition of $\mathbf{d}_\infty$.
\end{remark}

\subsection{Standing assumptions}
Fix $T>0$ and 
$p\ge 2$.
Let $q$ be defined by $\frac{1}{p}+\frac{1}{q}=1$.

Fix  an operator $A:[0,T]\times V\to V^\ast$ and a function $f^A\in L^1(0,T;\R_+)$. 
The following standing hypotheses (cf.~\cite{Liu11Nonlin,LiuRoeckner10JFA}) are always in force:

\textbf{H}($A$): (i) For every $x$, $v\in V$, the map $t\mapsto \langle A(t,x),v\rangle$ is measurable.

(ii) Local monotonicity: There is a constant $c_0\in\R$ and there 
are  locally bounded functions $\rho$ and $\eta$ from $V$ to $\R_+$ such that,
for each $x$, $y\in V$, we have
\begin{align*}
\langle A(t,x)-A(t,y),x-y\rangle\ge -(c_0+\rho(x)+\eta(y)) \abs{x-y}^2\qquad\text{a.e.~on $(0,T)$.}
\end{align*}
Moreover,
\begin{align}\label{E:Ac5}
\exists \beta>0:\,\forall x\in V:\,\rho(x)+\eta(x)\le \abs{c_0}(1+\norm{x}^p)(1+\abs{x}^\beta).
\end{align}

(iii) Hemicontinuity: For every $x$, $y$,  $v\in V$ and a.e.~$t\in (0,T)$, the map $s\mapsto \langle A(t,x+sy), v\rangle$,
$[0,1]\to\R$, is continuous.

(iv) Growth: There are constants   $c_1$, $\alpha\in\R_+$ such that, for all $x\in V$ and a.a~$t\in (0,T)$, we have
\begin{align*}
\norm{A(t,x)}_{\ast}\le \left(f^A(t)^{1/q}+c_1\norm{x}^{p-1}\right) (1+\abs{x}^\alpha)
\end{align*}

(v) Coercivity: There are constants  $c_2>0$ and $c_3\in\R_+$  such that, for all $x\in V$ and a.a~$t\in (0,T)$, we have
\begin{align*}
\langle A(t,x),x\rangle\ge c_2\norm{x}^p-c_3\abs{x}^2-f^A(t).
\end{align*}

\begin{remark}
Note that \cite{Liu11Nonlin} has different signs in the monotonicity hypothesis \textbf{H}($A$)~(ii)
and the coercivity hypothesis \textbf{H}($A$)~(iv).
The reason is that in \cite{Liu11Nonlin} equations
of the form $x^\prime(t)=A(t,x(t))+f(t)$ are considered,
whereas
 we deal with equations of the form $x^\prime(t)+A(t,x(t))=f(t)$. 
 \end{remark}

Also, fix a multifunction $F:[0,T]\times H \rightsquigarrow H$ with non-empty, convex, and closed values.
The following  hypotheses  
 are always in force.

\textbf{H}($F$): (i) {\color{black} $F$ is u.s.c.~in the sense of Definition~2.6.2 of \cite{Carja07book}, i.e.,
for each $(t,x)\in [0,T]\times H$ and for each open neighborhood $O$ of $F(t,x)$,
there is an open neighborhood $U$ of $(t,x)$ such that $(s,y)\in U$ implies
$F(s,y)\subset O$.}

(ii) There is a constant $c_F\ge 0$ such that, for a.e.~$t\in (0,T)$ and all $x\in H$,
\begin{align}\label{E:HFii}
\abs{F(t,x)}:=\sup\{\abs{y}:\,y\in F(t,x)\}\le c_F\cdot(1+\abs{x}).
\end{align}

\subsection{Function spaces}
{\color{black}We introduce several function spaces, which are frequently used throughout this work.}
{\color{black}  To this end, we employ the sets
\begin{align}\label{E:Wpq}
W_{pq} (t_0,T):=\{x\in L^p(t_0,T;V):\,x^\prime \in L^q(t_0,T;V^\ast)\},\quad\text{$t_0\in [0,T)$,}
\end{align}
where $x^\prime$ denotes the generalized derivative of $x$
(see p.~4 in \cite{HP2} for more details).}


\begin{definition}
Let $(t_0,x_0)\in [0,T)\times C([0,T],H)$. 

(i) Define trajectory spaces related to our multifunction $F$ by
\begin{equation}\label{E:XF}
\begin{split}
\mathcal{X}^F(t_0,x_0)&:=\{ x\in C([0,T],H)\text{ with $x\vert_{(t_0,T)}\in W_{pq}(t_0,T)$:}\\
&\quad\quad\, \text{$\exists f^x\in L^2(t_0,T;H):\, x=x_0$}\text{ on $[0,t_0]$ and }\\
&\qquad\qquad\, x^\prime(t)+A(t,x(t))=f^x(t)\in F(t,x(t))\text{ a.e.~on $(t_0,T)$\}.}\\
\mathcal{X}^F(T,x_0)&:=\{x_0\}
\end{split}
\end{equation}

(ii) Given $E\subset H$, put
\begin{equation}\label{E:XE}
\begin{split}
\mathcal{X}^E(t_0,x_0)&:=\{ x\in C([0,T],H):\,x\vert_{(t_0,T)}\in W_{pq}(t_0,T)\text{ and   $\exists f\in E_{L_2}$:}\\
&\quad\quad x^\prime(t)+A(t,x(t))=f(t)\text{ a.e.~on $(t_0,T)$,}\,
x=x_0\text{ on $[0,t_0]$}\}.
\end{split}
\end{equation}

(iii) Given $c\ge 0$, define the multifunction $B_c:[0,T]\times C([0,T],H)\rightsquigarrow H$  by
\begin{align}\label{E:XB0}
B_c(t_0,x_0):=\{y\in H:\,\abs{y}\le  c\cdot (1+\sup_{t\le t_0}\abs{x_0(t)})\}
\end{align}
and put
\begin{equation}\label{E:XB}
\begin{split}
\mathcal{X}^{B_c}(t_0,x_0)&:=\{ x\in C([0,T],H)\text{ with $x\vert_{(t_0,T)}\in W_{pq}(t_0,T)$:}\\
&\quad\quad \exists f^x\in L^2(t_0,T;H):\,x=x_0\text{ on $[0,t_0]$ and}\\
&\qquad\qquad
 x^\prime(t)+A(t,x(t))=f^x(t)\in B_c(t,x)\text{ a.e.~on $(t_0,T)$}
\}.
\end{split}
\end{equation}
\end{definition}

\begin{remark}
There should be no danger of confusion between \eqref{E:XF} and \eqref{E:XE}, as in \eqref{E:XF} the superscript $F$ denotes a set-valued function
whereas in \eqref{E:XE} the superscript $E$ denotes a set.
\end{remark}

We omit the proof of the following a-priori estimates, as it is nearly identical to the proof of Lemma~A.1 in \cite{BK18JFA}
(the only 
differences are our 
weaker growth and coercivity hypotheses 
\textbf{H}($A$) (iv) and (v) compared to the ones in section~2.2 of \cite{BK18JFA}). 

\begin{lemma}\label{L:apriori}
Let $c\ge 0$ and $r\ge 0$.  Then there is a constant
$C=C(c,r)>0$ such that, for all
$(t_0,x_0)\in [0,T)\times C([0,T],H)$   with 
{\color{black} $\sup_{t\le t_0}\abs{x_0(t)}\le r$}
and for all  $x\in\mathcal{X}^{B_c}(t_0,x_0)$ with corresponding
function $f^x \in L^2(t_0,T;H)$, we have
\begin{align*}
\norm{x}_\infty+\norm{x}_{W_{pq}(t_0,T)}+\norm{\hat{A}x}_{L^q(t_0,T;V^\ast)}+\norm{f^x}_{L^2(t_0,T;H)}\le C.
\end{align*}
Here, $\hat{A}x:[t_0,T]\to V^\ast$ is defined by $\hat{A} x(t)=A(t,x(t))$.
\end{lemma}

\begin{lemma}\label{L:precompact}
Let $c\ge 0$ and $(t_0,x_0)\in [0,T)\times C([0,T],H)$.
Then $\mathcal{X}^{B_c}(t_0,x_0)$ is precompact in $C([0,T],H)$.
In particular, for each sequence $(x_n)_n$ in $\mathcal{X}^{B_c}(t_0,x_0)$   with corresponding sequence $(f^{x_n})_n$ in $L^2(t_0,T;H)$,
there exist a pair $(x,f)\in C([0,T],H) \times L^2(t_0,T;H)$ 
with $x\vert_{(t_0,T)}\in W_{pq}(t_0,T)$
and  a subsequence $(x_{n_k})_k$  of $(x_n)_n$  such that 
$x_{n_k}\to x$ in $C([0,T],H)$,
$x_{n_k} \xrightarrow{w} x$ in $W_{pq}(t_0,T)$, $f^{x_{n_k}}\xrightarrow{w} f$ in $L^2(t_0,T;H)$, and
$x^\prime(t)+A(t,x(t))=f(t)$ a.e.~on $(t_0,T)$.
\end{lemma}

\begin{proof}
The proof is essentially the same as the proof of  Lemma~A.2 in \cite{BK18JFA} with the exception of the following two points (both due to $A$ only
being locally monotone as required in  \textbf{H}($A$)~(ii), which is weaker than monotonicty).
First, instead of the usual monotonicity trick (p.~474 in
\cite{ZeidlerIIB}), one should use the ``modified monotonicity trick" (Lemma~2.5 in \cite{Liu11Nonlin}).
The second point is about the convergence of $(x_{n_k})$ to $x$ in $C([0,T],H)$ assuming 
that \eqref{E:Ac5} holds and that the remaining parts of the lemma's statement are already established.
We show this convergence by slightly adjusting part~(ii) of the proof of Theorem~1.1 in \cite{Liu11Nonlin}.
Comparing with the proof of Lemma~A.2 in \cite{BK18JFA}, we replace (A.9) in \cite{BK18JFA} by
\begin{equation}\label{E:L:precompact1}
\begin{split}
\frac{1}{2}\abs{x_{n_k}(t)-x(t)}^2&=\int_{t_0}^t -\langle A(s,x_{n_k}(s))-A(s,x(s)),x_{n_k}(s)-x(s)\rangle\\
&\qquad\qquad + (f^{x_{n_k}}(s)-f^x(s),x_{n_k}(s)-x(s))\,ds\\
&\le \int_{t_0}^t 
(c_0+\rho(x_{n_k}(s))+\eta(x(s))) \abs{x_{n_k}(s)-x(s)}^2\,ds\\
&\qquad\qquad +2C\norm{x_{n_k}-x}_{L^2(t_0,T;H)}
\end{split}
\end{equation}
This is correct in our setting 
thanks to  \textbf{H}($A$)~(ii). The constant $C$ in \eqref{E:L:precompact1} is the same as in Lemma~\ref{L:apriori}.
Since \eqref{E:L:precompact1} holds for every $t\in [t_0,T]$, Gronwall's lemma together with
\eqref{E:Ac5} and Lemma~\ref{L:apriori}  yield
\begin{align*}
\sup_{t\in [t_0,T]} \abs{x_{n_k}(t)-x(t)}^2&\le  2C\norm{x_{n_k}-x}_{L^2(t_0,T;H)}
e^{\int_{t_0}^T 
(c_0+\rho(x_{n_k}(s))+\eta(x(s)))\,ds}\\ &\le 2C\tilde{C}\norm{x_{n_k}-x}_{L^2(t_0,T;H)}\to 0\text{ as $k\to\infty$}
\end{align*}
for some constant $\tilde{C}$ that is independent from $t$ and $k$. 
\end{proof}

\begin{theorem}\label{T:closure}
Fix $(t_\ast,x_\ast)\in [0,T]\times C([0,T],H)$. 
Let $(t_n,x_n)_n$ be a sequence in {\color{black}$[t_\ast,T]\times \mathcal{X}^{B_{(c_F)+1}}(t_\ast,x_\ast)$} 
that converges to some pair
$(t_0,x_0)\in [t_\ast,T]\times C([0,T],H)$.  Then every sequence $(\tilde{x}_n)_n$ with
$\tilde{x}_n\in\mathcal{X}^{F}(t_n,x_n)$ has a convergent subsequence with limit in 
$\mathcal{X}^F(t_0,x_0)$. 
\end{theorem}

\begin{proof}
Fix a sequence $(\tilde{x}_n)_n$ with $\tilde{x}_n\in\mathcal{X}^{F}(t_n,x_n)$.
By \eqref{E:HFii} in  \textbf{H}($F$)~(ii), we have 
{\color{black} $\mathcal{X}^F(t_n,x_n)\subset \mathcal{X}^{B_{(c_F)+1}}(t_\ast,x_\ast)$},
 i.e., $\tilde{x}_n=x_\ast$ on $[0,t_\ast]$ and
{\color{black}$\tilde{x}_n^\prime(t)+A(t,\tilde{x}_n(t))=f^{\tilde{x}_n}(t)\in B_{(c_F)+1}(t,\tilde{x}_n)$} a.e.~on $(t_\ast,T)$
for some $f^{\tilde{x}_n}\in L^2(t_\ast,T;H)$.
Hence, by Lemma~\ref{L:precompact}, $(t_n,x_n,\tilde{x}_n)_n$ has a subsequence,
which  we still denote by $(t_n,x_n,\tilde{x}_n)_n$, such that $\tilde{x}_n\to\tilde{x}_0$ in $C([0,T],H)$
and $f^{\tilde{x}_n}\xrightarrow{w} f^{\tilde{x}_0}$ in $L^2(t_\ast,T;H)$
for some pair $(\tilde{x}_0,f^{\tilde{x}_0})\in C([0,T],H)\times L^2(t_\ast,T;H)$ that additionally satisfies
$\tilde{x}_0^\prime(t)+A(t,\tilde{x}_0(t))=f^{\tilde{x}_0}(t)$ a.e.~on $(t_\ast,T)$,
$\tilde{x}_0=x_\ast$ on $[0,t_\ast]$, and $\tilde{x}_0\vert_{(t_\ast,T)}\in W_{pq}(t_\ast,T)$. Clearly,
we also have $\tilde{x}_0=x_0$ on $[0,t_0]$.
Next, let $m\in\N$ with $1/m\in (0,T-t_0)$. Without loss of generality, assume that $t_n<t_0+1/m$ for all $n\ge m$, i.e.,
$f^{\tilde{x}_n}(t)\in F(t,\tilde{x}_n(t))$ a.e.~on $(t_0+1/m,T)$ whenever $n\ge m$. 
Hence, by Lemma~2.6.2 in \cite{Carja07book}, $f^{\tilde{x}_0}(t)\in F(t,\tilde{x}_0(t))$ a.e.~on $(t_0+1/m,T)$.
Finally, the arbitrariness of $m$ yields $\tilde{x}_0\in\mathcal{X}^F(t_0,x_0)$.
\end{proof}

We will very often use the following regularity condition\footnote{
Recall the constant $c_F$ from \eqref{E:HFii} in  \textbf{H}($F$)~(ii) and $\mathcal{X}^{B_c}$ from \eqref{E:XB}.
} for a function $u:[0,T]\times C([0,T],H)\to\R\cup\{+\infty\}$
that is slightly weaker than lower semicontinuity:
\begin{equation}\label{u:LSC}
\begin{split}
\forall (t_\ast,x_\ast)\in [0,T]\times C([0,T],H):\,
u\vert_{[t_\ast,T]\times\mathcal{X}^{{\color{black}B_{(c_F)+1}}}(t_\ast,x_\ast)}\text{ is l.s.c.}
\end{split}
\end{equation}

\section{Viability}
Given $(t_0,x_0)\in [0,T)\times C([0,T],H)$, consider the evolution inclusion
\begin{equation}\label{E:Inclusion}
\begin{split}
x^\prime(t)+A(t,x(t))&\in F(t,x(t))\text{ a.e.~on $(t_0,T)$,}\\
x&=x_0\text{ on $[0,t_0]$.}
\end{split}
\end{equation}

\begin{definition}
A set $K\subset H$ is \emph{viable for \eqref{E:Inclusion}} if, for every $(t_0,x_0)\in [0,T)\times C([0,T],H)$
with $x_0(t_0)\in K$, there exists an $x\in\mathcal{X}^F(t_0,x_0)$ such $x(t)\in K$ for all $t\in [t_0,T]$.\footnote{
Recall the definition of $\mathcal{X}^F(t_0,x_0)$
from \eqref{E:XF}.
}
\end{definition}

The next definition is adapted from Remark 10.1 of \cite{Carja09TAMS}. 

\begin{definition}\label{D:QTS1}
Fix $K\subset H$.
Let $(t_0,x_0)\in [0,T]\times C([0,T],H)$ with $x_0(t_0)\in K$.
A non-empty set $E\subset H$ is
\emph{$A$-quasi-tangent to $K$ at $(t_0,x_0)$} if
there are sequences $(\delta_n)_n$ in $\R_+$ with $\delta_n\downarrow 0$,
$(b_n)_n$ in $L^2(0,T;H)$, 
$(p_n)_n$ in $B(0,1/n)_{L^2}$,
 and $(x_n)_n$ in $C([0,T],H)$ with  $x_n\vert_{(t_0,T)}\in W_{pq}(t_0,T)$ such that
\begin{align*}
x_n^\prime(t)+A(t,x_n(t))&=b_n(t)+p_n(t)\text{ a.e.~on  $(t_0,{\color{black} T})$,}\\ 
b_n(t)& \in E \text{ a.e.~on $(t_0,t_0+\delta_n)$,}\\
x_n&=x_0\text{ on $[0,t_0]$, and}\\
x_n(t_0+\delta_n)&\in K.
\end{align*}
We write $\mathcal{QTS}^A_K(t_0,x_0)$
for
the class of all $A$-quasi-tangent sets to $K$ at $(t_0,x_0)$.
\end{definition}

{\color{black} Our main result on necessary and sufficient conditions of viability for  \eqref{E:Inclusion}
can be found below in subsection~\ref{SS:NecSuff}.}

{\color{black} Next, we study  viability for an extension of \eqref{E:Inclusion}, which is useful for 
applications to partial differential equations in section~\ref{S:HJB}.}
Given $(t_0,x_0,y_0)\in [0,T)\times C([0,T],H)\times\R$, consider the system
\begin{equation}\label{E:Inclusion2}
\begin{split}
s^\prime(t)&=1\qquad\qquad\qquad\text{on $(t_0,T)$,}\\
x^\prime(t)+A(t,x(t))&\in F(t,x(t))\,\,\,\qquad\text{a.e.~on $(t_0,T)$,}\\
y^\prime(t)&=0\qquad\qquad\qquad\text{on $(t_0,T)$,}\\
(s(t_0),x\vert_{[0,t_0]},y(t_0))&=(t_0,x_0\vert_{[0,t_0]},y_0).
\end{split}
\end{equation}

\begin{definition}
The epigraph of a function $u:[0,T]\times C([0,T],H)\to\R\cup\{+\infty\}$
  is \emph{viable for \eqref{E:Inclusion2}} if, for every $(t_0,x_0,y_0)\in\mathrm{epi}\,u$ with $t_0<T$,
   there exists an $x\in\mathcal{X}^F(t_0,x_0)$ such  that $(t,x,y_0)\in\mathrm{epi}\,u$
  for all $t\in [t_0,T]$.
\end{definition}

The next definition extends Definition~\ref{D:QTS1} in the spirit of Definition~4.4 of 
\cite{Keller24}. 
\begin{definition}\label{D:QTS2}
Fix  $u:[0,T]\times C([0,T],H)\to\R\cup\{+\infty\}$. 
Let $(t_0,x_0,y_0)\in\mathrm{epi}\,u$ with $t_0<T$.
A  non-empty set $E\subset H$ is
\emph{$A$-quasi-tangent to $\mathrm{epi}\,u$ at $(t_0,x_0,y_0)$} if
there are sequences $(\delta_n)_n$ in $\R_+$ with $\delta_n\downarrow 0$,
 $(b_n)_n$ in $L^2(0,T;H)$, 
$(p_n)_n$ in $B(0,1/n)_{L^2}$,
and $(x_n)_n$ in $C([0,T],H)$ with  $x_n\vert_{(t_0,T)}\in W_{pq}(t_0,T)$ such that
\begin{equation}\label{E:QTS2}
\begin{split}
x_n^\prime(t)+A(t,x_n(t))&=b_n(t)+p_n(t)\text{ a.e.~on  $(t_0,{\color{black} T})$,}\\ 
b_n(t)&\in E \text{ a.e.~on $(t_0,t_0+\delta_n)$,}\\
x_n&=x_0\text{ on $[0,t_0]$, and}\\
u(t_0+\delta_n,x_n)&\le y_0+\delta_n/n.
\end{split}
\end{equation}
We write $\mathcal{QTS}^A_{\mathrm{epi}\,u}(t_0,x_0,y_0)$
for the class of all $A$-quasi-tangent sets to $\mathrm{epi}\,u$ at $(t_0,x_0,y_0)$.
\end{definition}

\begin{lemma}\label{L:Ku}
Let $K\subset H$. Consider $u:[0,T]\times C([0,T],H)\to\R$  defined by
$u(t,x):=-\bfone_K(x(t))$.\footnote{If $K$ is closed, then $u$ is l.s.c.} 
Then the following holds:

(i) The set $K$ is viable for \eqref{E:Inclusion}
if and only if $\mathrm{epi}\,u$ is viable for \eqref{E:Inclusion2}.

(ii) Let $E$ be a non-empty subset of $H$. Then  
$E\in \mathcal{QTS}^A_K(t_0,x_0)$ for each $(t_0,x_0)\in [0,T)\times C([0,T],H)$ with $x_0(t_0)\in K$ if and only if
$E\in \mathcal{QTS}^A_{\mathrm{epi}\,u}(t_0,x_0,y_0)$ for each $(t_0,x_0,y_0)\in\mathrm{epi}\,u$ with $t_0<T$.
\end{lemma}
\begin{proof}
(i) First, assume that  $K$ is viable for \eqref{E:Inclusion}. Fix $(t_0,x_0,y_0)\in\mathrm{epi}\,u$ with $t_0<T$. If $y_0\ge 0$, then
$(t,x,y_0)\in\mathrm{epi}\, u$ on $[t_0,T]$ for all $x\in\mathcal{X}^F(t_0,x_0)$. 
If $y_0\in [-1,0)$, then $u(t_0,x_0)=-1$, i.e., $x_0\in K$ and thus there is an $x\in\mathcal{X}^F(t_0,x_0)$ such
that $x(t)\in K$ on $[t_0,T]$, which yields $y_0\ge u(t,x)$ on $[t_0,T]$.

Next, assume that  $\mathrm{epi}\,u$ is viable for \eqref{E:Inclusion2}. Fix $(t_0,x_0)\in [0,T)\times C([0,T],H)$ with $x_0(t_0)\in K$.
Since $(t_0,x_0,u(t_0,x_0))\in\mathrm{epi}\,u$, there is an $x\in\mathcal{X}^F(t_0,x_0)$
such that $-1=u(t_0,x_0)\ge u(t,x)=-\bfone_K(x(t))$ on $[t_0,T]$, i.e., $K$ is viable for \eqref{E:Inclusion}.

(ii) 
 First, let $(t_0,x_0,y_0)\in\mathrm{epi}\,u$ with $t_0<T$ and assume
 that  $E\in \mathcal{QTS}^A_K(t_\ast,x_\ast)$ for each $(t_\ast,x_\ast)\in [0,T)\times C([0,T],H)$ with $x_\ast(t_\ast)\in K$.
 If $x_0(t_0)\in K$, then  $E\in \mathcal{QTS}^A_K(t_0,x_0)$ with corresponding 
 sequences $(\delta_n)_n$, $(b_n)_n$, $(p_n)_n$,
and $(x_n)_n$ from Definition~\ref{D:QTS1}, which satisfy  $x_n(t_0+\delta_n)\in K$ for all $n\in\N$,
 and consequently, we have  $-1=u(t_0+\delta_n,x_n)\le y_0+\delta_n/n$ because $y_0\ge -1$ due to $x_0(t_0)\in K$, i.e.,
$E\in \mathcal{QTS}^A_{\mathrm{epi}\,u}(t_0,x_0,y_0)$.
 If $x_0(t_0)\not\in K$, then $y_0\ge 0$ and thus
 $u(t_0+\delta_n,x_n)\le y_0+\delta_n/n$ from \eqref{E:QTS2} is automatically satisfied for all possible sequences $(\delta_n)_n$ and $(x_n)_n$ from
 Definition~\ref{D:QTS2} 
  (we only need to invoke standard existence results for evolution equations
 such as Theorem~1.1 in \cite{Liu11Nonlin}
  to ensure the existence of at least one such
 sequence $(x_n)_n$). In any case, we have  
$E\in \mathcal{QTS}^A_{\mathrm{epi}\,u}(t_0,x_0,y_0)$.
 
 Now, let $(t_0,x_0)\in [0,T)\times C([0,T],H)$ with $x_0(t_0)\in K$ and assume
 that $E\in \mathcal{QTS}^A_{\mathrm{epi}\,u}(t_0,x_0,u(t_0,x_0))$
 with corresponding sequences $(\delta_n)_n$, $(b_n)_n$, $(p_n)_n$,
and $(x_n)_n$ from Definition~\ref{D:QTS2}.
By \eqref{E:QTS2} and $x_0(t_0)\in K$, we have
 $u(t_0+\delta_n,x_n)\le u(t_0,x_0)+\delta_n/n=-1+\delta_n/n$ for each $n\in\N$.
 Since $\delta_n\downarrow 0$, we have $u(t_n,x_n)<0$  and thus $x_n(t_0+\delta_n)\in K$ for all sufficiently large $n\in\N$.
Hence, $E\in \mathcal{QTS}^A_K(t_0,x_0)$.
\end{proof}

\subsection{{\color{black} Necessary condition for viability for  \eqref{E:Inclusion2}}}
\begin{theorem}\label{T:Necessary2}
Let 
$u:[0,T]\times C([0,T],H)\to\R\cup\{+\infty\}$. 
Suppose that 
$\mathrm{epi}\,u$ is viable for \eqref{E:Inclusion2}.
Then, for every $(t_0,x_0,y_0)\in \mathrm{epi}\,u$ with $t_0<T$, we have
\begin{equation}\label{E:T:Necessary2}
\begin{split}
F(t_0,x_0(t_0))\in \mathcal{QTS}^A_{\mathrm{epi}\,u}(t_0,x_0,y_0).
\end{split}
\end{equation}
\end{theorem}

\begin{proof}
We proceed as in the proof of Theorem~10.1 in \cite{Carja09TAMS}.
Fix $(t_0,x_0,y_0)\in \mathrm{epi}\,u$ with $t_0<T$.
 Since $\mathrm{epi}\,u$ is viable for \eqref{E:Inclusion2}, there is 
 an $x\in\mathcal{X}^F(t_0,x_0)$ with corresponding selector $f^x$ of $F(\cdot,x)$
 such that $(t,x,y_0)\in \mathrm{epi}\,u$  for all $t\in [t_0,T]$.
 Let $n\in\N$. By   upper semicontinuity of $F$ and continuity of $x$, there is a $\delta_n\in (0,1/n]$ such that, for all $t\in [t_0,t_0+\delta_n]$, we have
 \begin{align*}
  F(t,x(t))\subset F(t_0,x_0(t_0))+B(0,1/n).
 \end{align*}
 Thus, by Lemma~10.1 and Remark~10.5, both 
 in \cite{Carja09TAMS}, we have 
  $f^x=b_n+p_n$ a.e.~on {\color{black} $(t_0,T)$} 
  for  some
  {\color{black} $b_n\in L^2(0,T;H)$  and $p_n\in  B(0,1/n)_{L^2}$ with
  $b_n(t)\in F(t_0,x_0)$ a.e.~on $(t_0,t_0+\delta_n)$}.
  Since $(\delta_n,b_n,p_n,x_n)$ with $x_n=x$ satisfies \eqref{E:QTS2}, we have 
\eqref{E:T:Necessary2}.
\end{proof}

\subsection{{\color{black} Sufficient condition for viability for  \eqref{E:Inclusion2}}}
See section~\ref{SS:Proof:Sufficient2} for the proof of the next result.
\begin{theorem}\label{T:Sufficient2}
Let $u:[0,T]\times C([0,T],H)\to\R\cup\{+\infty\}$ 
satisfy \eqref{u:LSC}.
If, for every $(t_0,x_0,y_0)\in \mathrm{epi}\,u$ with $t_0<T$, we have
\begin{align*}
 F(t_0,x_0(t_0))\in \mathcal{QTS}^A_{\mathrm{epi}\,u}(t_0,x_0,y_0),
\end{align*}
then $\mathrm{epi}\,u$ is viable for \eqref{E:Inclusion2}.
\end{theorem}

\subsection{{\color{black} Necessary and sufficient condition for viability for  \eqref{E:Inclusion}}}\label{SS:NecSuff}
The next result follows immediately from  Theorems~\ref{T:Necessary2}, \ref{T:Sufficient2}, and Lemma~\ref{L:Ku}.  
\begin{cor}\label{T:Nagumo1}
Let $K$ be a closed subset of $H$. 
Then $K$ is viable if and only if, for every $(t_0,x_0)\in [0,T)\times C([0,T],H)$
with $x_0(t_0)\in K$, we have
\begin{align*}
F(t_0,x_0(t_0))\in \mathcal{QTS}^A_K(t_0,x_0).
\end{align*}
\end{cor}

\begin{remark}
 If $K=H$, then, by  Theorem~1.1 of \cite{Liu11Nonlin},  which guarantees existence of solutions for locally monotone evolution equations,
 we have
  $F(t_0,x_0(t_0))\in \mathcal{QTS}^A_K(t_0,x_0)$.
Thus {\color{black}Corollary}~\ref{T:Nagumo1} is also an existence result for evolution inclusions.
\end{remark}

\subsection{{\color{black} $\eps$-approximate solutions}}
The next definition is an appropriate modification
of Definition~12.1 in \cite{Carja09TAMS}
(see also Definition~3 in \cite{Goreac11} and Definition~4.9 in \cite{Keller24} for stochastic cases).

\begin{definition} 
Let $\eps\in (0,1]$,
  $u:[0,T]\times C([0,T],H)\to\R\cup\{+\infty\}$,  
and $(t_0,x_0,y_0)\in\mathrm{epi}\, u$ with $t_0<T$. 
We call  a quintuple $(\tau,\varrho,f,g,x)$
an \emph{$\eps$-approximate solution of \eqref{E:Inclusion2} for $\mathrm{epi}\,u$ starting at $(t_0,x_0,y_0)$}  if the following holds:

(i) $\tau\in (t_0,T]$.

(ii)  $\varrho:[t_0,\tau]\to [t_0,\tau]$ is non-decreasing and we have
 $t-\eps\le\varrho(t)\le t$ for all $t\in [t_0,\tau]$ as well as $\varrho(\tau)=\tau$.

(iii) $f\in L^2({\color{black} 0},T;H)$ {\color{black} with $f(t)=0$ a.e.~on $(\tau,T)$.} 

(iv) $g\in B(0,\eps)_{L^2}$\text{ with $g(t)=0$ a.e.~on $(\tau,T)$.}

(v) $x\in C([0,T],H)$ 
satisfies $x\vert_{(t_0,T)}\in W_{pq}(t_0,T)$ and
\begin{equation}\label{E:Inclusion:Eps}
\begin{split}
x^\prime(t)+A(t,x(t))&=f(t)+g(t)\text{ a.e.~on $(t_0,T)$,}\\
f(t) &\in F({\color{black}\varrho(t)},x(\varrho(t)))
 \text{ a.e.~on $(t_0,\tau)$,}\\
x&=x_0\text{ on $[0,t_0]$}.
\end{split}
\end{equation}

(vi) $u(\varrho(t),x)\le y_0+\eps\cdot(t-t_0)$ for all $t\in [t_0,\tau]$.
\end{definition}

The next result is an adjustment of Lemma~12.1 in \cite{Carja09TAMS}
(see Remark~\ref{R:Approximate} for more details).

\begin{proposition}\label{P:Approximate}
Let $\eps>0$. Let $u:[0,T]\times C([0,T],H)\to\R\cup\{+\infty\}$ 
satisfy \eqref{u:LSC}.
Suppose that, for every $(t,x,y)\in\mathrm{epi}\,u$ with $t<T$, 
\begin{align}\label{E:P:Approximate}
F(t,x(t))\in \mathcal{QTS}^A_{\mathrm{epi}\,u}(t,x,y).
\end{align}
Then, for every $(t_0,x_0,y_0)\in \mathrm{epi}\,u$ with $t_0<T$,
 there 
exists an $\eps$-approximate solution  $(\tau,\varrho,f,g,x)$ of 
\eqref{E:Inclusion2} for $\mathrm{epi}\,u$ starting at $(t_0,x_0,y_0)$ such  that $\tau=T$.
\end{proposition}

\begin{proof}
Fix $(t_0,x_0,y_0)\in \mathrm{epi}\,u$ with $t_0<T$.
Denote by $\mathcal{S}$ the set of $\eps$-approximate solution  of 
\eqref{E:Inclusion2} for $\mathrm{epi}\,u$ starting at $(t_0,x_0,y_0)$.
Given  $\mathfrak{s}_1=(\tau_1,\varrho_1,f_1,g_1,x_1)$, $\mathfrak{s}_2=(\tau_2,\varrho_2,f_2,g_2,x_2)\in \mathcal{S}$,
 we write $\mathfrak{s}_1\preceq\mathfrak{s}_2$ if $\tau_1\le\tau_2$, $(\varrho_1,x_1)=(\varrho_2,x_2)$ on $[t_0,\tau_1]$,
 and $(f_1,g_1)=(f_2,g_2)$ a.e.~on $(t_0,\tau_1)$. Note that  $\preceq$ defines a preorder on $\mathcal{S}$.
In Steps~2 and 3 below, we shall also use the function $\mathcal{N}:\mathcal{S}\to [t_0,T]$ defined by
$\mathcal{N}(\tau,\varrho,f,g,x):=\tau$. Moreover, we  shall call  $\mathfrak{s}\in\mathcal{S}$ 
an  \emph{$\mathcal{N}$-maximal} element of $\mathcal{S}$
if $\mathfrak{s}\preceq\mathfrak{s}_+\in\mathcal{S}$ yields $\mathcal{N}(\mathfrak{s})=\mathcal{N}(\mathfrak{s}_+)$
(section~2.1 in  \cite{Carja07book}).

\textit{Step 1 (existence of $\eps$-approximate solutions).}
By \eqref{E:P:Approximate} and Definition~\ref{D:QTS2}, there are $\delta\in (0,\eps]$,
$f\in L^2(0,T;H)$,  $g\in B(0,\eps)_{L_2}$, and $x\in C([0,T],H)$ with
$x\vert_{(t_0,T)}\in W_{pq}(t_0,T)$  such that $u(t_0+\delta,x)\le y_0+\eps\delta$ and
\begin{align*}
x^\prime(t)+A(t,x(t))&=f(t)+g(t)\text{ a.e.~on $(t_0,T)$,}\\
f(t)&\in F(t_0,x_0(t_0))\text{ a.e.~on $(t_0,t_0+\delta)$,}\\
x&=x_0\text{ on $[0,t_0]$.} 
\end{align*}
Put $\tau:=t_0+\delta$.
Define $\varrho:[t_0,\tau]\to [t_0,\tau]$ by $\varrho(t):=t_0$ for $t\in [t_0,\tau)$ and
$\varrho(\tau):=\tau$. Finally, consider a solution\footnote{
Such a solution  exists thanks to Theorem~1.1 in \cite{Liu11Nonlin}.
} $\tilde{x}\in C([0,T],H)$ with $\tilde{x}\vert_{(\tau,T)}\in W_{pq}(\tau,T)$  of
\begin{align*}
\tilde{x}^\prime(t)+A(t,\tilde{x}(t))&=0\text{ a.e.~on $(\tau,T)$,}\\
\tilde{x}&=x\text{ on $[0,\tau]$} 
\end{align*}
 Then $(\tau,\varrho,\bfone_{(0,\tau)}\cdot f, \bfone_{(0,\tau)}\cdot g, \bfone_{[0,\tau]}\cdot x+\bfone_{(\tau,T]}\cdot \tilde{x})\in\mathcal{S}$.


\textit{Step 2 (existence of maximal $\eps$-approximate solutions).}
Consider an increasing sequence 
 $(\mathfrak{s}_n)_{n\ge 1}=(\tau_n,\varrho_n,f_n,g_n,x_n)_{n\ge 1}$   in $\mathcal{S}$. 
 We  show  that this sequence is bounded from above.
 Then the Brezis--Browder principle
 (Theorem~2.1.1 in  \cite{Carja07book})
  will yield the existence of an $\mathcal{N}$-maximal
 element in $\mathcal{S}$.

 To establish boundedness of $(\mathfrak{s}_n)_n$, note 
 first that
 \begin{equation}\label{E:Step2:taufg}
 \begin{split}
 \tau_n&\uparrow \tau:=\sup_m\tau_m,\\
 f_n (t)&\to  f(t):=\sum_{m=1}^\infty \bfone_{(\tau_{m-1},\tau_m]}(t)\, f_m(t)\text{ a.e.~on $(0,T)$},\\
 g_n (t)&\to  g(t):=\sum_{m=1}^\infty \bfone_{(\tau_{m-1},\tau_m]}(t)\, g_m(t)\text{ a.e.~on $(0,T)$},
 \end{split}
 \end{equation}
 where $\tau_0:=t_0$.
 Next, note that $x_n\in \mathcal{X}^{B_c}(t_0,x_0)$ with
 $c=\eps+c_F$ for each $n\in\N$ according to \eqref{E:Inclusion:Eps}, \eqref{E:HFii}, 
 \eqref{E:XB0}, and \eqref{E:XB}.
 Thus, by  Lemma~\ref{L:precompact}, $(\mathfrak{s}_n)_n$  has 
  a subsequence $(\mathfrak{s}_{n_k})_k$ such that
 \begin{equation}\label{E:Step2:x:f}
 \begin{split}
 x_{n_k}&\to x\text{ in $C([0,T],H)$ as well as in $W_{pq}(t_0,T)$ and}\\
 f_{n_k}+g_{n_k}&\xrightarrow{w} \tilde{f}\text{ in $L^2(t_0,T;H)$ as $k\to\infty$}
 \end{split}
 \end{equation}
 for some $(x,\tilde{f})\in C([0,T],H)\times L^2(t_0,T;H)$ with $x=x_0$ on $[0,t_0]$,
 $x\vert_{(t_0,T)}\in W_{pq}(t_0,T)$, and $x^\prime(t)+A(t,x(t))=\tilde{f}(t)$ a.e.~on $(t_0,T)$.
 By \eqref{E:Step2:taufg}, $\tilde{f}=f+g$ a.e.~on $(t_0,T)$.
 Next, define $\varrho:[t_0,\tau]\to [t_0,\tau]$ by 
 \begin{align*}
 \varrho(t):=\bfone_{\{t_0\}}(t)\cdot \varrho_1(t)+\sum_{n=1}^\infty \bfone_{(\tau_{n-1},\tau_n]\cap\{\tau\}^c}(t)\cdot\varrho_n(t)
 +\bfone_{\{\tau\}}(t)\cdot \tau.
 \end{align*}
Now, we show that $f(t)\in F(\varrho(t),x(\varrho(t))$ a.e.~on $(t_0,\tau)$. To this end, fix an arbitrary
 $\delta\in (0,\tau-t_0)$ and an $m\in\N$ such that $\tau_m>\tau-\delta$. As 
 $f_n(t)\in F(\varrho_n(t),x(\varrho_n(t))$ a.e.~on $(t_0,\tau-\delta)$ for each $n\ge m$ and as
 $F$ is u.s.c.~and is non-empty, convex, and closed valued, 
we only need, by Lemma~2.6.2 in \cite{Carja07book},
 to show that
 \begin{align}\label{E:Step2:rho:x}
 (\varrho_{n_k}(t),x_{n_k}(\varrho_{n_k}(t))\to(\varrho(t),x(\varrho(t))\text{ a.e.~on $(t_0,\tau-\delta)$ as $k\to\infty$}
 \end{align}
 in order to obtain $f(t)\in F(\varrho(t),x(\varrho(t))$~a.e.~on $(t_0,\tau-\delta)$.
Indeed, as we have  $(\varrho_n(t),x_n(\varrho_n(t))=(\varrho(t),x_n(\varrho(t))$  for every $t\in [t_0,\tau-\delta]$ 
and $n\ge m$,
 \eqref{E:Step2:x:f} yields  \eqref{E:Step2:rho:x}.
Thus, as  $\delta$ was arbitrary in  $(0,\tau-t_0)$,  
$f(t)\in F(\varrho(t),x(\varrho(t))$~a.e.~on $(t_0,\tau)$.
Also note  that  $x_{n_k}=x$ on $[0,\tau_{n_k}]$ for each $k\in\N$, as $(\mathfrak{s}_n)_n$ is increasing
and because of \eqref{E:Step2:x:f}.
 Thus  $u(\varrho(t),x)\le y_0+\eps \cdot(t-t_0)$  for all $t\in [t_0,\tau)$.  
By \eqref{u:LSC} and \eqref{E:HFii},
\begin{align*}
u(\varrho(\tau),x)=u(\tau,x)\le \varliminf_k u(\tau_{n_k},x_{n_k})= \varliminf_k u(\varrho_{n_k}(\tau_{n_k}),x_{n_k})
\le y_0+\eps\cdot(\tau-t_0).
\end{align*}
We  conclude that $\mathfrak{s}:=(\tau,\varrho,f,g,x)\in\mathcal{S}$ and 
$\mathfrak{s}_{n_k} \preceq\mathfrak{s}$ for all $k\in\N$.
Also note that, for each $m\in\N$, there is a $k\in\N$ with $n_k\ge m$ and thus
$\mathfrak{s}_{m} \preceq\mathfrak{s}_{n_k}\preceq\mathfrak{s}$.
Consequently, as pointed out at the beginning of Step~2, there exists an $\mathcal{N}$-maximal
 element in $\mathcal{S}$.


\textit{Step 3 (extension step).}
Let $\mathfrak{s}_0=(\tau_0,\varrho_0,f_0,g_0,x_0)$ be an 
$\mathcal{N}$-maximal element of $\mathcal{S}$ with $\tau_0<T$.
In particular, we have $(\tau_0,x_0,y_0+\eps\cdot(\tau_0-t_0))\in \mathrm{epi}\,u$.
Thus, by \eqref{E:P:Approximate} and Definition~\ref{D:QTS2}, there is a  quadruple
\begin{align*}
(\delta,b,p,x)\in (0,\eps]\times L^2(0,T;H)\times B(0,\eps)_{L^2}\times C([0,T],H)
\end{align*}
with $x\vert_{(\tau_0,T)}\in W_{pq}(t_0,T)$, $x^\prime(t)+A(t,x(t))=b(t)+p(t)$ a.e.~on $(\tau_0,T)$,
$b(t)\in F(\tau_0,x_0(\tau_0))$ a.e.~on $(\tau_0,\tau_0+\delta)$, $x=x_0$ on $[0,\tau_0)$,
and 
\begin{align*}
u(\tau_0+\delta,x)\le y_0+\eps\cdot(\tau_0-t_0)+\eps\cdot\delta.
\end{align*}
Next, just as in Step~1, let $\tilde{x}\in C([0,T],H)$  with $\tilde{x}\vert_{(\tau_0+\delta,T)}\in W_{pq}
(\tau_0+\delta,T)$ be a solution of $\tilde{x}^\prime(t)+A(t,\tilde{x}(t))=0$ a.e.~on $(\tau_0+\delta,T)$
with initial condition $\tilde{x}=x$ on $[0,\tau_0+\delta]$.
Then
\begin{align*}
\mathfrak{s}_0 &\preceq \mathfrak{s}_+:=
(\tau_0+\delta,\bfone_{[t_0,\tau_0]}(t)\cdot\varrho_0+
\bfone_{(\tau_0,\tau_0+\delta)}\cdot \tau_0+\bfone_{\{\tau_0+\delta\}}\cdot (\tau_0+\delta),\\
&\qquad\qquad
\bfone_{(0,\tau_0)}\cdot f+\bfone_{(\tau_0,\tau_0+\delta)}\cdot b,
\bfone_{(0,\tau_0)}\cdot g_0+\bfone_{(\tau_0,\tau_0+\delta)}\cdot p,\\
&\qquad\qquad
\bfone_{[0,\tau_0]}\cdot x_0+\bfone_{(\tau_0,\tau_0+\delta]}\cdot x+\bfone_{(\tau_0+\delta,T]}\cdot \tilde{x})\in\mathcal{S}
\end{align*}
but $\mathcal{N}(\mathfrak{s}_0)\neq \mathcal{N}(\mathfrak{s}_+)$, 
which contradicts the $\mathcal{N}$-maximality of $\mathfrak{s}_0$. Thus, $\tau_0=T$.
\end{proof}

\begin{remark}\label{R:Approximate}
The main difference  of the previous proof   to corresponding places in \cite{Carja09TAMS}
(see section 12 therein)
 can be found in Step~2
of the proof of Proposition~\ref{P:Approximate}, where we use a compactness argument to
obtain an extension of the ``$x$-components" of $(\tau_n,\varrho_n,f_n,g_n,x_n)_n$. 
This is possible thanks to our operator $A$ being coercive, which is not assumed in \cite{Carja09TAMS},
Note  that in \cite{Carja09TAMS} the operator $A$  generates a $C_0$-semigroup. This allows for different arguments
in \cite{Carja09TAMS}.
\end{remark}

\subsection{Proof of Theorem~\ref{T:Sufficient2}}\label{SS:Proof:Sufficient2}
Fix $(t_0,x_0,y_0)\in\mathrm{epi}\,u$ with $t_0<T$. Then,
for each $n\in\N$, there exists, by Proposition~\ref{P:Approximate},
an $(1/n)$-approximate solution $\mathfrak{s}_n=(\tau_n,\varrho_n,f_n,g_n,x_n)$ 
of \eqref{E:Inclusion2} for $\mathrm{epi}\,u$ starting at $(t_0,x_0,y_0)$ with $\tau_n=T$.
Note that
\begin{align}\label{E1:Proof:Sufficient2}
\varrho_{n}(t)&\to t\text{ on $[t_0,T]$  and $g_n(t)\to 0$ a.e.~on $(t_0,T)$.}
\end{align}
Lemma~\ref{L:precompact} yields (cf.~Step~2 of the proof of  Proposition~\ref{P:Approximate})
 the existence of a subsequence 
$(\mathfrak{s}_{n_k})_k$ of $(\mathfrak{s}_n)_n$ such that
$x_{n_k}\to x$ in $C([0,T],H)$ as well as in $W_{pq}(t_0,T)$ and $f_{n_k}+g_{n_k}\xrightarrow{w} f$ in  $L^2(t_0,T;H)$
for some $(x,f)\in C([0,T],H)\times L^2(t_0,T;H)$ with $x\vert_{(t_0,T)}\in W_{pq}(t_0,T)$,
$x=x_0$ on $[0,t_0]$, and $x^\prime(t)+A(t,x(t))=f(t)$ a.e.~on $(t_0,T)$.
We can  invoke now Lemma~2.6.2 in \cite{Carja07book} (cf.~Step~2 of the proof of  Proposition~\ref{P:Approximate})
to obtain $f(t)\in F(t,x(t))$ a.e.~on $(t_0,T)$ because 
$f_{n_k}\xrightarrow{w} f$ in  $L^2(t_0,T;H)$ and $(\varrho_{n_k}(t),x_{n_k}(\varrho_{n_k}(t)))\to (t,x(t))$ a.e.~on $(t_0,T)$
due to  \eqref{E1:Proof:Sufficient2} and to
\begin{equation}\label{E2:Proof:Sufficient2}
\begin{split}
&\norm{x(\cdot\wedge t)-x_{n_k}(\cdot\wedge\varrho_{n_k}(t))}_\infty\\ &\qquad\le
\norm{x(\cdot\wedge t)-x(\cdot\wedge\varrho_{n_k}(t))}_\infty+
\norm{x(\cdot\wedge\varrho_{n_k}(t))-x_{n_k}(\cdot\wedge\varrho_{n_k}(t))}_\infty\\
&\qquad\le \omega(1/n_k)+\norm{x-x_{n_k}}_\infty\to 0\text{ for every $t\in [t_0,T]$,}
\end{split}
\end{equation}
where $\omega$ is a modulus of continuity of $x$. Finally, by  \eqref{u:LSC} and \eqref{E:HFii},
\begin{align*}
 u(t,x)
\le\varliminf_k u(\varrho_{n_k}(t),x_{n_k})\le \varliminf_k [y+(1/n_k)\cdot (t-t_0)]=y,
\end{align*}
as $\mathbf{d}_\infty(\varrho_{n_k}(t),x_{n_k};t,x)\to 0$ due to \eqref{E2:Proof:Sufficient2}.
This concludes the proof. \qed

\section{Hamilton--Jacobi--Bellman equations and optimal control}\label{S:HJB}
In this section, we additionally assume that $F$ has bounded values. 

{\color{black} Fix 
an l.s.c.~\emph{terminal cost function}
 $h:C([0,T],H)\to\R\cup\{+\infty\}$.  We consider the following  optimal control problem. 
 \begin{quotation}
 $\textbf{(OC).}$ Given $(t_0,x_0)\in [0,T)\times C([0,T],H)$, find a solution $\tilde{x}(\cdot)$ of  evolution inclusion \eqref{E:1st}
 with initial condition $\tilde{x}=x_0$ on $[0,t_0]$, i.e., find an element $\tilde{x}\in\mathcal{X}^F(t_0,x_0)$,
 such that
 \begin{align*}
 h(\tilde{x})=\inf\{h(x):\,x\in\mathcal{X}^F(t_0,x_0)\}.
 \end{align*}
 \end{quotation}}
{\color{black} The} \emph{value function} 
$v:[0,T]\times C([0,T],H)\to\R\cup\{+\infty\}$ {\color{black} of  $\textbf{(OC)}$ 
 is} defined by
\begin{align}\label{E:v}
v(t_0,x_0):=\inf\{h(x):\,x\in\mathcal{X}^F(t_0,x_0)\}. 
\end{align}
Formally, $v$ is a solution of the terminal-value problem
\begin{equation}\label{E:TVP}
\begin{split}
&-\partial_t u+\langle A(t,x(t)),\partial_x u\rangle -\inf_{f\in F(t,x(t))} (f,\partial_x u)=0\text{ in $[0,T)\times C([0,T],H)$},\\
&u(T,x)=h(x)\text{ on $C([0,T],H)$.}
\end{split}
\end{equation}

{\color{black} 
\begin{remark}
Fix a topological space $P$, a function $f:[0,T]\times H\times P\to H$, and denote by
$\mathcal{A}$ the set of all Borel measurable functions from $[0,T]$ to $P$.
Consider a Mayer problem in the  
more usual form of finding a control 
$\tilde{a}\in\mathcal{A}$
such that
\begin{align*}
h(x^{t_0,x_0,\tilde{a}})=\inf \{h(x^{t_0,x_0,a}):\, a\in\mathcal{A}\}\quad
\text{given $(t_0,x_0)\in [0,T)\times C([0,T],H)$.} 
\end{align*}
Here, $x=x^{t_0,x_0,a}$ solves 
\begin{align*}
x^\prime(t)+A(t,x(t))&=f(t,x(t),a(t))\text{ a.e.~on $(t_0,T)$},\\
x&=x_0\text{ on $[0,t_0]$}.
\end{align*}
 Such a problem can, under appropriate assumptions, be formulated as our optimal control problem $\textbf{(OC)}$
(for more details, see section~1 in  \cite{Frankowska93SICON} for the finite-dimensional case 
and Chapter IV in \cite{HP2}  for the infinite-dimensional case).  
\end{remark}
}

\subsection{Properties of the value function}
\begin{theorem}\label{T:v}
The following holds:  

 (i) For every  $(t_0,x_0)\in [0,T)\times C([0,T],H)$,
there exists an optimal trajectory $x\in\mathcal{X}^F(t_0,x_0)$, i.e., $v(t_0,x_0)=h(x)$.

 (ii) $v$ satisfies \eqref{u:LSC}.

(iii) $v$ satisfies the dynamic programming principle, i.e., if $0\le t_0<t\le T$, then
\begin{align}\label{E:DPP1}
v(t_0,x_0)&=v(t,\tilde{x}_0)\text{ for some $\tilde{x}_0\in\mathcal{X}^F(t_0,x_0)$ independent of $t$ and}\\ \label{E:DPP2}
v(t_0,x_0)&\le v(t,x)\text{ for all $x\in\mathcal{X}^F(t_0,x_0)$.}
\end{align}
\end{theorem}

\begin{proof}
(i) follows from the compactness of $\mathcal{X}^F(t_0,x_0)$ 
(Theorem~\ref{T:closure})
and from $h$ being l.s.c.

(ii) Fix $(t_\ast,x_\ast)\in [0,T]\times C([0,T],H)$ and 
$(t_0,x_0)\in [t_\ast,T]\times \mathcal{X}^{{\color{black}B_{(c_F)+1}}}(t_\ast,x_\ast)$.
Let $(t_n,x_n)_n$ be a sequence in 
$ [t_\ast,T]\times \mathcal{X}^{{\color{black}B_{(c_F)+1}}}(t_\ast,x_\ast)$
 that converges to $(t_0,x_0)$.
By part~(i), there is a sequence $(\tilde{x}_n)_n$ with $\tilde{x}_n\in\mathcal{X}^F(t_n,x_n)$ such that
$v(t_n,x_n)=h(\tilde{x}_n)$ for each $n\in\N$. By Theorem~\ref{T:closure}, this sequence has a subsequence
$(\tilde{x}_{n_k})_k$ that converges to some $\tilde{x}_0\in\mathcal{X}^F(t_0,x_0)$. Hence, thanks to
$h$ being l.s.c.~and \eqref{E:v}, 
\begin{align*}
\varliminf_k v(t_{n_k},x_{n_k})=\varliminf_k h(\tilde{x}_{n_k})\ge h(\tilde{x}_0)\ge v(t_0,x_0),
\end{align*}
i.e., $v$ satisfies \eqref{u:LSC}.

(iii) By part~(i),  $v(t_0,x_0)=h(\tilde{x}_0)$ for some $\tilde{x}_0\in\mathcal{X}^F(t_0,x_0)$ that does not depend on $t$.
Let $x\in \mathcal{X}^F(t_0,x_0)$. Then, again by part~(i), we have
$v(t,x)=h(\tilde{x})$ for some $\tilde{x}\in\mathcal{X}^F(t,x)\subset\mathcal{X}^F(t_0,x_0)$. Thus, by \eqref{E:v},
$v(t_0,x_0)=h(\tilde{x}_0)\le h(\tilde{x})$, i.e., \eqref{E:DPP2} holds. Finally, note that, by part~(i),
$v(t,\tilde{x}_0)=h(\tilde{x}_1)$ for some $\tilde{x}_1\in\mathcal{X}^F(t,\tilde{x}_0)$. Since $\tilde{x}_0\in\mathcal{X}^F(t,\tilde{x}_0)$,
we have, by \eqref{E:v}, $v(t,\tilde{x}_0)=h(\tilde{x}_1)\le h(\tilde{x}_0)=v(t_0,x_0)$. Together with \eqref{E:DPP2}, \eqref{E:DPP1} follows.
\end{proof}


\subsection{Quasi-contingent solutions}
We define  the \emph{$A$-contingent epiderivatives} $D^A_\uparrow$
of a function $u:[0,T]\times C([0,T],H)\to\R\cup\{+\infty\}$ at a point $(t,x)\in\mathrm{dom}\,u$ with $t<T$ in a multi-valued
direction $E\subset H$ by\footnote{Recall \eqref{E:XE} for the definition of $\mathcal{X}^{E+B(0,\eps)}$.}
\begin{align*}
D^A_\uparrow\, u(t,x)\,(E):=\sup_{\eps>0}\,
\inf&\Biggl\{
\frac{u(t+\delta,\tilde{x})-u(t,x)}{\delta}:
\delta\in (0,\eps],
\,\tilde{x}\in\mathcal{X}^{E+B(0,\eps)}(t,x)\Biggr\}. 
\end{align*}


\begin{remark}\label{R1:lowerUpper}
Similar (second-order) contingent epiderivatives with  directions being sets of stochastic processes
have been introduced in  Definition~5.1 of \cite{Keller24}. Also note
that a finite-dimensional counterpart of $D^A_\uparrow\, u(t,x)\,(E)$ with $A=0$ plays an
important role in \cite{GP23JFA} (see subsection~4.2 therein) in order to establish equivalence
between viscosity and minimax solutions for path-dependent Hamilton--Jacobi equations.
Consider also the slightly different  \emph{lower and upper derivatives} 
\begin{align*}
d^A_-\,u(t,x)\,(E)&:=\sup_{\eps>0}\,
\inf\left\{
\varliminf_{\delta\downarrow 0}\frac{u(t+\delta,\tilde{x})-u(t,x)}{\delta}:\,\tilde{x}\in\mathcal{X}^{E+B(0,\eps)}(t,x)\right\},\\
d^A_+\,u(t,x)\,(E)&:=\inf_{\eps>0}\,
\sup\left\{
\varlimsup_{\delta\downarrow 0}\frac{u(t+\delta,\tilde{x})-u(t,x)}{\delta}:\,\tilde{x}\in\mathcal{X}^{E+B(0,\eps)}(t,x)\right\},
\end{align*}
which are more closely linked to path derivatives
(see Remark~4.2 in \cite{GP23JFA} and (8.4) in \cite{Lukoyanov03a}).
Appropriate counterparts of $d^A_-$ and $d^A_+$ (in finite dimensions with $A=0$)
 are used in the theory of minimax solutions of path-dependent 
Hamilton--Jacobi equations (see, e.g., 
\cite{GLP21AMO,GP23JFA, Lukoyanov06_DiffIneq,Lukoyanov03a}).
In this work, we mainly use 
$D^A_\uparrow$ 
 because
of their connection to quasi-contingent sets (see Lemma~\ref{L:Epiderivative} below, which more or less  corresponds 
 to the relationship of contingent derivatives with contingent cones  in Proposition~6.1.4 of \cite{AubinFrankowska_setValued}).
\end{remark}

\begin{lemma}\label{L:Epiderivative}
Let 
$u:[0,T]\times C([0,T],H)\to\R\cup\{+\infty\}$. Then, for every
$(t_0,x_0)\in\mathrm{dom}\,u$ with $t_0<T$, we have 
\begin{align*}
D^A_\uparrow\,u(t_0,x_0)\,(F(t_0,x_0(t_0)))\le 0\Longleftrightarrow F(t_0,x_0(t_0))\in\mathcal{QTS}^A_{\mathrm{epi}\,u}(t_0,x_0,u(t_0,x_0)).
\end{align*}
\end{lemma}
\begin{proof}
First, let $D^A_\uparrow\,u(t_0,x_0)\,(F(t_0,x_0(t_0)))\le 0$, i.e.,  for each $n\in\N$, there are $x_n\in\mathcal{X}^{F(t_0,x_0(t_0))+B(0,1/n)}(t_0,x_0)$ and 
$\delta_n\in (0,1/n]$ with 
$u(t_0+\delta_n,x_n)\le u(t_0,x_0)+\delta_n/n$.
Then, together with Lemma~10.1 
 of \cite{Carja09TAMS} (cf.~the proof of Theorem~\ref{T:Necessary2}),  we can deduce that, for each $n\in\N$, we have
 $x_n^\prime(t)+A(t,x_n(t))=b_n(t)+p_n(t)$~a.e.~on $(t_0,T)$ for some $b_n\in F(t_0,x_0(t_0))_{L_2}$
 and $p_n\in B(0,1/n)_{L_2}$. Thus, recalling Definition~\ref{D:QTS2}, we can see that
 $F(t_0,x_0(t_0))\in\mathcal{QTS}^A_{\mathrm{epi}\,u}(t_0,x_0,u(t_0,x_0))$ holds.

Next, assume that 
 $F(t_0,x_0(t_0))\in\mathcal{QTS}^A_{\mathrm{epi}\,u}(t_0,x_0,u(t_0,x_0))$
holds with corresponding sequences $(\delta_n)_n$, $(b_n)_n$, $(p_n)_n$,
and $(x_n)_n$ from Definition~\ref{D:QTS2},
which satisfy \eqref{E:QTS2}   with $E=F(t_0,x_0(t_0))$ and $y_0=u(t_0,x_0)$.
Fix an arbitrary $f\in F(t_0,x_0(t_0))$. Next, for each $n\in\N$,
 fix an arbitrary $x_n^f\in\mathcal{X}^{\{f\}}(t_0+\delta_n,x_n)$
(which is possible by Theorem~1.1 in \cite{Liu11Nonlin}).
Then, for each $n\in\N$, the functions
\begin{align*}
\tilde{b}_n&:=\bfone_{[t_0,t_0+\delta_n]}\,b_n+\bfone_{[0,t_0)\cup(t_0+\delta_n,T]}\,f\quad\text{and}\quad
\tilde{x}_n:=\bfone_{[0,t_0+\delta_n]}\,x_n+\bfone_{(t_0+\delta_n,T]}\,x^f_n
\end{align*}
satisfy
  $\tilde{b}_n+p_n\in (F(t_0,x_0(t_0))+B(0,1/n))_{L^2}$ and, by   Theorem~1.16  on p.~6 in \cite{HP2}, 
   $\tilde{x}_n\in\mathcal{X}^{F(t_0,x_0(t_0))+B(0,1/n)}(t_0,x_0)$, i.e., 
   $D^A_\uparrow\,u(t_0,x_0)\,(F(t_0,x_0(t_0)))\le 0$ 
by \eqref{E:QTS2}.
\end{proof}

\begin{definition}\label{D:QC:sol}
Let $u:[0,T]\times C([0,T],H)\to\R\cup\{+\infty\}$ be a function.

(i) We call $u$ a \emph{quasi-contingent supersolution of \eqref{E:TVP}} if $u$ 
satisfies \eqref{u:LSC}, 
 $u(T,\cdot)\ge h$, and, for all
$(t_0,x_0)\in \mathrm{dom}\,u$ with $t_0<T$, we have
\begin{align}\label{E:QC:Super}
D^A_\uparrow u(t_0,x_0)\,(F(t_0,x_0(t_0)))\le 0.
\end{align} 


(ii) We call $u$ an \emph{l.s.c.\footnote{Note that only appropriate restrictions of  $u$ are required
to be l.s.c. For details, see \eqref{u:LSC}.
}~quasi-contingent subsolution of \eqref{E:TVP}} if $u$ 
satisfies \eqref{u:LSC},
$u(T,\cdot)\le h$, and, for all
$(t_\ast,x_\ast)\in [0,T)\times C([0,T],H)$,  $t_0\in (t_\ast,T]$, and $x_0\in\mathcal{X}^{F(t_\ast,x_\ast)}(t_\ast,x_\ast)$ with
$(t_0,x_0)\in \mathrm{dom}\,u$, we have
\begin{align}\label{E:QC:Sub}
\varliminf_{\delta\downarrow 0}\frac{u(t_0-\delta,x_0)-u(t_0,x_0)}{\delta}\le 0.
\end{align} 


(iii) We call $u$ an \emph{l.s.c.~ quasi-contingent solution of \eqref{E:TVP}} if $u$ is a quasi-contingent super- and an l.s.c.~quasi-contingent subsolution of \eqref{E:TVP}.
\end{definition}

\begin{theorem}\label{T:value:contingent}
The value function $v$ is an l.s.c.~quasi-contingent solution of \eqref{E:TVP}.
\end{theorem}
\begin{proof}
(i) (Regularity). By Theorem~\ref{T:v}, $v$ satisfies \eqref{u:LSC}.

(ii) (Quasi-contingent supersolution property).  We proceed very similarly to the proof of Theorem~8.1 in \cite{Lukoyanov03a}.
Fix $(t_0,x_0)\in\mathrm{dom}\,v$ with $t_0<T$. By \eqref{E:DPP1}, there is
an $x\in\mathcal{X}^F(t_0,x_0)$ with selector $f^x$ such that
$\varliminf_{\delta\downarrow 0} \delta^{-1}\,[v(t_0+\delta,x)-v(t_0,x_0)]\le 0$.
Since, in addition, for every $n\in\N$, there is a $\delta_n\in (0,1/n]$ such
that $f^x=b_n+p_n$ a.e.~on $(t_0,t_0+\delta_n)$ for some $b_n\in F(t_0,x_0(t_0))_{L^2}$ and $p_n\in B(0,1/n)_{L^2}$
(cf.~the proof of Theorem~\ref{T:Necessary2}),
we can deduce that there is a sequence $(x_n)_n$ such that, for every $n\in\N$, we have 
$x_n\in\mathcal{X}^{F(t_0,x_0(t_0))+B(0,1/n)}(t_0+\delta_n,x)$ 
and  $x_n\in\mathcal{X}^{F(t_0,x_0(t_0))+B(0,1/n)}(t_0,x_0)$ as well. Note that  $x_n=x$ on $[0,t_0+\delta_n]$. Hence, 
\begin{align*}
&\sup_{n\in\N}\,\inf\left\{\frac{v(t_0+\delta,\tilde{x})-v(t_0,x_0)}{\delta}:\,0<\delta<\frac{1}{n},\,{\tilde{x}\in\mathcal{X}^{F(t_0,x_0(t_0))+B(0,1/n)}(t_0,x_0)}
\right\}\\
&\le\lim_{n\to\infty}\,\inf\left\{
\varliminf_{\delta\downarrow 0}\frac{v(t_0+\delta,\tilde{x})-v(t_0,x_0)}{\delta}:\,{\tilde{x}\in\mathcal{X}^{F(t_0,x_0(t_0))+B(0,1/n)}(t_0,x_0)}
\right\}\\
&\le \varlimsup_{n\to\infty}
\varliminf_{\delta\downarrow 0}\frac{v(t_0+\delta,x_n)-v(t_0,x_0)}{\delta}=\varlimsup_{n\to\infty}
\varliminf_{\delta\downarrow 0}\frac{v(t_0+\delta,x)-v(t_0,x_0)}{\delta}\le 0, 
\end{align*}
which yields  \eqref{E:QC:Super}.

(iii) (L.s.c.~quasi-contingent subsolution property). By \eqref{E:DPP2}, we have \eqref{E:QC:Sub}.
\end{proof}

\begin{theorem}\label{T:u:QC:super}
Let $u$ be a quasi-contingent supersolution of \eqref{E:TVP}. Let $(t_0,x_0)\in \mathrm{dom}\,u$. 
Then there is an $x\in\mathcal{X}^F(t_0,x_0)$ 
with  $u(t,x)\le u(t_0,x_0)$ for all $t\in [t_0,T]$.
\end{theorem}
\begin{proof}
For every $(t,x,y)\in\mathrm{epi}\,u$ with $t<T$, we have
$D^A_\uparrow u(t,x)\,(F(t,x(t)))\le 0$, which, by Lemma~\ref{L:Epiderivative}, implies
 $F(t,x(t))\in\mathcal{QTS}^A_{\mathrm{epi}\,u}(t,x,y)$.
 Thus Theorem~\ref{T:Sufficient2} yields the viability of $\mathrm{epi}\,u$ for \eqref{E:Inclusion2}.
 Taking $y=u(t_0,x_0)$ concludes the proof.
\end{proof}

\begin{theorem}\label{T:u:QC:sub}
Let $u$ be an l.s.c.~{\color{black}quasi-}contingent 
subsolution of \eqref{E:TVP}. Let $(t_0,x_0)\in \mathrm{dom}\,u$. 
Then   $u(t,x)\ge u(t_0,x_0)$ for all $t\in [t_0,T]$ and  $x\in\mathcal{X}^F(t_0,x_0)$.
\end{theorem}
\begin{proof}
One can proceed (nearly exactly) as in the proof of in Lemma~8.2 in \cite{BK22SICON}.
\end{proof}



\begin{theorem}\label{T:u:QC:comparison}
Let $u_-$ be an  l.s.c.~quasi-contingent subsolution
of \eqref{E:TVP}
and let $u_+$ be a quasi-contingent supersolution of \eqref{E:TVP}.
Then $u_-\le u_+$.
\end{theorem}
\begin{proof}
Let $(t_0,x_0)\in [0,T]\times C([0,T],H)$.
By Theorem~\ref{T:u:QC:super}, there is an $x\in\mathcal{X}^F(t_0,x_0)$  with
$u_+(T,x)\le u_+(t_0,x_0)$.  Thus, by  Definition~\ref{D:QC:sol} and Theorem~\ref{T:u:QC:sub},  
\begin{align*}
u_-(t_0,x_0)\le u_-(T,x)\le h(x)\le u_+(T,x)\le u_+(t_0,x_0)
\end{align*}
This concludes the proof.
\end{proof}

{\color{black} The next result follows immediately from Theorems~\ref{T:value:contingent} and \ref{T:u:QC:comparison}.
\begin{cor}
The value function $v$ is the unique l.s.c.~quasi-contingent solution of \eqref{E:TVP}.
\end{cor}
}

\subsection{Viscosity solutions}\label{SS:Viscosity}
First, we introduce spaces of smooth functions on path spaces.
To this end, we rephrase Definition~2.16 in \cite{BK18JFA}.

\begin{definition} 
Fix $t_0\in [0,T)$. We denote by  $\mathcal{C}_V^{1,1}([t_0,T]\times C([0,T],H))$ the set of 
all continuous functions $\varphi:[t_0,T]\times C([0,T],H)\to\R$ for which there exist
continuous functions $\partial_t\varphi:[t_0,T]\times C([0,T],H)\to\R$
and $\partial_x\varphi:[t_0,T]\times C([0,T],H)\to H$, which we call 
\emph{path derivatives} of $\varphi$, such that, for each
$t_1$, $t_2\in [t_0,T]$ with $t_1<t_2$, and  each $x\in C([0,T],H)$
with $x\vert_{(t_1,t_2)}\in W_{pq}(t_1,t_2)$, we have $x(t)\in V$ implies $\partial_t v(t,x(t))\in V$ a.e.~on $(t_1,t_2)$
and
\begin{align*}
\varphi(t_2,x)-\varphi(t_1,x)=\int_{t_1}^{t_2} \partial_t\varphi(t,x)+\langle x^\prime(t),\partial_x\varphi(t,x)\rangle\,dt.
\end{align*}
\end{definition}

Now, we are able to define test function spaces needed for our definition of viscosity solutions.
\begin{definition}
Given $E\subset H$, $u:[0,T]\times C([0,T],H)\to\R\cup\{+\infty\}$, and $(t_0,x_0)\in [0,T)\times C([0,T],H)$, put
\begin{align*}
\overline{\Phi}_+^E\,u(t_0,x_0)&:=\{\varphi\in \mathcal{C}_V^{1,1}([t_0,T]\times C([0,T],H)):\\
&\qquad\qquad\exists\eps>0:\,\forall t\in [t_0,t_0+\eps]:\,\forall x\in\mathcal{X}^{E+B(0,\eps)}(t_0,x_0):\\
&\qquad\qquad\qquad\qquad 0=(\varphi-u)(t_0,x_0)\ge (\varphi-u)(t,x)\},\\
\overline{\Phi}_-\,u(t_0,x_0)&:=\{\varphi\in \mathcal{C}_V^{1,1}([0,t_0]\times C([0,T],H)):\,
\exists\eps>0:\,\forall t\in [t_0-\eps,t_0]:\\ 
&\qquad\qquad\qquad\qquad 0=(\varphi-u)(t_0,x_0)\ge (\varphi-u)(t,x_0)\}.
\end{align*}
\end{definition}

\begin{definition}\label{D:Viscosity}
Let $u:[0,T]\times C([0,T],H)\to\R\cup\{+\infty\}$ be a function.

(i) We call $u$ a \emph{viscosity supersolution} of  
\begin{equation}\label{E:plusTVP}
\begin{split}
&-\partial_t u+\langle A(t,x(t)),\partial_x u\rangle -\inf_{f\in F(t,x(t))} (f,\partial_x u)=0\text{ on $[0,T)\times C([0,T],H)$}
\end{split}
\end{equation}
if $u$ satisfies \eqref{u:LSC} and,
 for every
$(t_0,x_0)\in [0,T)\times C([0,T],H)$  with  $(t_0,x_0)\in\mathrm{dom}\,u$ and every test function 
$\varphi\in\overline{\Phi}_+^{F(t_0,x_0(t_0))}\,u(t_0,x_0)$ with corresponding number  
$\eps>0$,
there exists 
 an $x\in\mathcal{X}^{F(t_0,x_0(t_0))+B(0,\eps)}(t_0,x_0)$ such that 
\begin{equation}\label{E:Viscosity:Super}
\begin{split}
-\partial_t \varphi(t_0,x_0)&+\varliminf_{\delta\downarrow 0}\frac{1}{\delta}
\int_{t_0}^{t_0+\delta} \langle A(t,x(t)),\partial_x\varphi(t,x)\rangle\,dt
\\&
-\inf_{f\in F(t_0,x_0(t_0))} (f,\partial_x\varphi(t_0,x_0))\ge 0.
\end{split}
\end{equation}

(ii) We call $u$ a  \emph{viscosity supersolution}  of
\begin{equation}\label{E:minusTVP}
\begin{split}
&\partial_t u-\langle A(t,x(t)),\partial_x u\rangle +\inf_{f\in F(t,x(t))} (f,\partial_x u)=0\text{ on $(0,T]\times C([0,T],H)$}
\end{split}
\end{equation}
if $u$  satisfies \eqref{u:LSC} and,  
for all
$(t_\ast,x_\ast)\in [0,T)\times C([0,T],H)$,   all $t_0\in (t_\ast,T]$, all $x_0\in\mathcal{X}^F(t_\ast,x_\ast)$ with selector $t\mapsto f^{x_0}(t)\in F(t,x(t))$
and  with $(t_0,x_0)\in\mathrm{dom}\,u$, 
and  all $\varphi\in\overline{\Phi}_-\,
u(t_0,x_0)$, we have 
\begin{align}\label{E:Viscosity:SuperMinus}
\partial_t \varphi(t_0,x_0)&+\varlimsup_{\delta\downarrow 0}\frac{1}{\delta}
\int_{t_0-\delta}^{t_0} \langle -A(t,x_0(t))+f^{x_0}(t),\partial_x\varphi(t,x_0)\rangle\,dt
\ge 0. 
\end{align}

(iii) We call $u$  
an \emph{l.s.c.\footnote{Note that only appropriate restrictions of  $u$ are required
to be l.s.c. For details, see \eqref{u:LSC}.
}~viscosity solution}  (or \emph{bilateral supersolution}) of \eqref{E:TVP}
if $u(T,\cdot)=h$ and if $u$ is a viscosity supersolution of \eqref{E:plusTVP} as well as of
\eqref{E:minusTVP}.
\end{definition}

\begin{remark}
Roughly speaking, satisfying the  viscosity supersolution property  for \eqref{E:minusTVP} with the test function spaces
$\overline{\Phi}_-\,u(t_0,x_0)$ can be thought of as
satisfying the viscosity supersolution property for all linear equations
of the form
\begin{align*}
\partial_t u(t_0,x_0)-\langle A(t_0,x_0(t_0)),\partial_x u(t_0,x_0)\rangle + (f,\partial_x u(t_0,x_0))=0
\end{align*}
for every $f\in F(t_0,x_0(t_0))$, $(t_0,x_0)\in (0,T]\times C([0,T],H)$, i.e., formally
\begin{align*}
\partial_t u(t_0,x_0)-\langle A(t_0,x_0(t_0)),\partial_x u(t_0,x_0)\rangle + 
\inf_{f\in F(t_0,x_0(t_0))}(f,\partial_x u(t_0,x_0))\ge 0
\end{align*}
holds in the viscosity sense. 
Hence, identifying the selectors $f^x$ in \eqref{E:XF} with admissible controls $t\mapsto a(t)$, we can say that
for an l.s.c.~viscosity solution $u$ of \eqref{E:TVP}  the following holds
in the viscosity sense for suitably defined 
 operators $\mathcal{L}_a^+$ and $\mathcal{L}_a^-$:
\begin{align*}
-\inf_{a\in F(t_0,x(t_0))}\mathcal{L}_a^+\,u(t_0,x_0)&\ge 0,\qquad (t_0,x_0)\in [0,T)\times C([0,T],H),\\
\mathcal{L}_a^- u(t_0,x_0)&\ge 0\qquad\text{for all $t\mapsto a(t)\in F(t,x_0(t))$,}\\
&\qquad\qquad\text{$(t_0,x_0)\in (t_\ast,T]\times\mathcal{X}^F(t_\ast,x_\ast)$,}\\
&\qquad\qquad\text{$(t_\ast,x_\ast)\in [0,T)\times C([0,T],H)$}
\end{align*}
(cf.~Definition 6.5 and Remark 6.6 in \cite{Keller24}).
\end{remark}

\begin{remark}\label{R:DPP}
{\color{black}
Typically (see \cite{BardiCapuzzoDolcetta}),
 bilateral supersolutions are related to a backward dynamic programming principle. E.g., consider
a problem  of the form 
\begin{align*}
\tilde{v}(t_0,x_0):=\inf\{\tilde{h}(x^{t_0,x_0,a}(T)):\,a\in\mathcal{A}),\quad(t_0,x_0)\in [0,T]\times\R,
\end{align*}
where $x=x^{t_0,x_0,a}$ solves $x^\prime(t)=f(t,x(t),a(t))$ on $(t_0,T)$ with initial condition $x(t_0)=x_0$. Then
$\tilde{v}(t_0,x_0)\le\tilde{v}(t,x^{t_0,x_0,a}(t))$ for all admissible controls $a\in\mathcal{A}$ and $t\in [t_0,T]$. Similarly, a backward version holds, i.e.,
$\tilde{v}(t,x_-^{t_0,x_0,a}(t)))\le\tilde{v}(t_0,x_0))$ for all $a\in\mathcal{A}$ and $t\in [0,t_0]$, where 
$x_-=x_-^{t_0,x_0,a}$ solves  $x_-^\prime(t)=f(t,x_-(t),a(t))$ on $(0,t_0)$ with terminal condition $x_-(t_0)=x_0$.
This backward version leads to viscosity supersolutions of the corresponding HJB equation (but with opposite sign).
In our context, the situation is slightly different. First, due to the operator $A$, we cannot expect solutions of backward evolution equations to exist, in general.
Second, in the path-dependent case, the terminal condition for a backward evolution equation with history-dependent data
means that $x^{t_0,x_0,a}\vert_{[0,t_0]}=x_0\vert_{[0,t_0]}$ should hold,\footnote{Note that here $x_0\in C([0,T],H)$.} which means
that we should only consider ``terminal data" $x_0$ that  already satisfy our evolution equation at least on some time interval $(t_0-\delta,t_0)$. This
naturally leads to 
Definition~\ref{D:Viscosity}~(ii) with the
test function spaces $\overline{\Phi}_-\,u(t_0,x_0)$.
}
\end{remark}

\subsubsection{Viscosity solutions: Comparison principle}

\begin{theorem}\label{T:QCsuper:Visc}
Let $u$ be viscosity supersolution of \eqref{E:plusTVP} with $u(T,\cdot)\ge h$. Then $u$
is a quasi-contingent supersolution of \eqref{E:TVP}.
\end{theorem}
\begin{proof}
Fix $(t_0,x_0)\in\mathrm{dom}\,u$ with $t_0<T$. Put $E:=F(t_0,x_0(t_0))$.
Assume that   
$D^A_\uparrow\,u(t_0,x_0)\,(E)> 0$,
i.e., there  are $c>0$ and $\eps>0$ such that
\begin{align*}
u(t_0+\delta,x)-u(t_0,x_0)>c\cdot\delta
\end{align*}
for all $\delta\in (0,\eps]$
and all $x\in\mathcal{X}^{E+B(0,\eps)}(t_0,x_0)$.
Now, we can easily (just as in the proof of Theorem~6.8 in \cite{Keller24})
obtain a test function 
$\varphi\in\overline{\Phi}_+^E\,u(t_0,x_0)$ that does not satisfy \eqref{E:Viscosity:Super},
i.e., we have a contradiction. Thus our assumption is wrong, i.e.,
$D^A_\uparrow\,u(t_0,x_0)\,(E)\le 0$ holds. This concludes the proof.
\end{proof}

\begin{theorem}\label{T:QCsub:Visc}
Let $u$ be viscosity supersolution of \eqref{E:minusTVP} with $u(T,\cdot)\le h$. Then $u$
is an l.s.c.~quasi-contingent subsolution of \eqref{E:TVP}.
\end{theorem}
\begin{proof}
Fix $(t_\ast,x_\ast)\in [0,T)\times C([0,T],H)$,  $t_0\in (t_\ast,T]$, and
$x_0\in\mathcal{X}^F(t_\ast,x_\ast)$ with selector $f^{x_0}\in F(\cdot,x(\cdot))$.
Suppose that $(t_0,x_0)\in\mathrm{dom}\,u$.
Assume that there are $c>0$ and $\eps>0$ such that
\begin{align*}
u(t_0-\delta,x_0)-u(t_0,x_0)>\delta\cdot c
\end{align*}
for all $\delta\in (0,\eps]$.
Define $\varphi: [0,t_0]\times C([0,T],H) \to\R$ by
\begin{align*}
\varphi(t,x):=u(t_0,x_0)+(t_0-t)\cdot c. 
\end{align*}
Then $\varphi\in \mathcal{C}_V^{1,1}([0,t_0]\times C([0,T],H))$ with
$\partial_x\varphi(t,x)=0$
and 
$\partial_t\varphi(t,x)=-c$, 
and thus $\varphi\in\overline{\Phi}_-\,u(t_0,x_0)$,
which contradicts \eqref{E:Viscosity:SuperMinus}. 
Hence our assumption is wrong and thus $u$ is an l.s.c.~quasi-contingent supersolution of \eqref{E:TVP}.
\end{proof}

The next result follows from Theorems~\ref{T:u:QC:comparison}, 
\ref{T:QCsuper:Visc}, and \ref{T:QCsub:Visc}.

\begin{theorem}\label{T:u:Visc:comparison}
If $u_-$ is 
a viscosity supersolution of \eqref{E:minusTVP} with $u_-(T,\cdot)\le h$ 
 and $u_+$ is 
  a viscosity supersolution of \eqref{E:plusTVP} with $u_+(T,\cdot)\ge h$, 
 then
 $u_-\le u_+$.
\end{theorem}

\subsubsection{Viscosity solutions: Existence and uniqueness}

\begin{theorem}\label{T:u:Visc:U}
The value function $v$ is the unique l.s.c.~viscosity solution of  
{\color{black} \eqref{E:TVP}.}
\end{theorem}

\begin{proof}
(i) \emph{Regularity}: By Theorem~\ref{T:v}, $v$ satisfies \eqref{u:LSC}. 

(ii) \emph{Viscosity supersolution property for \eqref{E:plusTVP}}: 
 Let
$(t_0,x_0)\in [0,T)\times C([0,T],H)$  with $(t_0,x_0)\in\mathrm{dom}\,v$ and 
$\varphi\in\overline{\Phi}_+^{F(t_0,x_0)}\,v(t_0,x_0)$ with corresponding 
$\eps>0$.
By \eqref{E:DPP1}, there is an $x\in\mathcal{X}^F(t_0,x_0)$ such that
$v(t,x)=v(t_0,x_0)$ for all $t\in [t_0,T]$.
From the proof of Theorem~\ref{T:Necessary2}, we can deduce that there are $\delta_0>0$
and $\tilde{x}\in\mathcal{X}^{F(t_0,x_0(t_0))+B(0,\eps)}(t_0,x_0)$ 
with corresponding selector 
$f^{\tilde{x}}=\tilde{b}+\tilde{p}$ a.e.~on $(t_0,t_0+\delta_0)$ for  some $\tilde{b}\in F(t_0,x(t_0))_{L^2}$ and $\tilde{p}\in  B(0,\eps)_{L^2}$ such that
that $x=\tilde{x}$ on $[0,t_0+\delta_0]$. Thus, thanks to  the upper semicontinuity of $F$, for every $n\in\N$, there are
 $\delta_n\in (0,\delta_0)$, $\tilde{b}_n\in F(t_0,x(t_0))_{L^2}$ and $\tilde{p}_n\in  B(0,1/n)_{L^2}$ such that, for every $\delta\in (0,\delta_n)$,  we have
\begin{align*}
0&=v(t_0+\delta,x)-v(t_0,x_0)=v(t_0+\delta,\tilde{x})-v(t_0,x_0)\\ 
&\ge \varphi(t_0+\delta,\tilde{x})-\varphi(t_0,x_0)\\
&\ge\int_{t_0}^{t_0+\delta} \partial_t\varphi(t,x)-\langle A(t,x(t)),\partial_x\varphi(t,x)\rangle+
 (\tilde{b}_n(t)+\tilde{p}_n(t),\partial_x\varphi(t,x))\,dt\\
&\ge\int_{t_0}^{t_0+\delta} \partial_t\varphi(t,x)-\langle A(t,x(t)),\partial_x\varphi(t,x)\rangle+
(\tilde{b}_n(t),\partial_x\varphi(t,x)-\partial_x\varphi(t_0,x_0))\\ &\qquad\qquad+
\inf_{f\in F(t_0,x(t_0))}(f,\partial_x\varphi(t_0,x_0))+(\tilde{p}_n(t),\partial_x\varphi(t,x))\,dt.
\end{align*}
Dividing by $\delta$ and applying $\varliminf_{\delta\downarrow 0}$ yields
\begin{align*}
&\partial_t \varphi(t_0,x_0)+\varliminf_{\delta\downarrow 0}\frac{1}{\delta}
\int_{t_0}^{t_0+\delta} \langle -A(t,x(t)),\partial_x\varphi(t,x)\rangle\,dt
+\inf_{f\in F(t_0,x_0(t_0))} (f,\partial_x\varphi(t_0,x_0))\\
&\le 
-\varliminf_{\delta\downarrow 0}\frac{1}{\delta}
\int_{t_0}^{t_0+\delta}[(\tilde{p}_n(t),\partial_x\varphi(t,x))
+
 (\tilde{b}_n(t),\partial_x\varphi(t,x)-\partial_x\varphi(t_0,x_0))]\,dt\\
&
\le \frac{1}{n}\cdot \sup_{t\in [t_0,T]} \abs{\partial_x\varphi(t,x)}+\abs{F(t_0,x_0)}\cdot\varlimsup_{\delta\downarrow 0} \frac{1}{\delta}\int_{t_0}^{t_0+\delta}  
\abs{\partial_x\varphi(t,x)-\partial_x\varphi(t_0,x_0))}\,dt\\
&=\frac{1}{n}\cdot \sup_{t\in [t_0,T]} \abs{\partial_x\varphi(t,x)}
\end{align*}
thanks to the superadditivitiy of $\varliminf$ and the boundedness of $F(t_0,x_0)$.
Since $n\in\N$ was arbitrary, 
we have
\eqref{E:Viscosity:Super}.

(iii) \emph{Viscosity supersolution property for \eqref{E:minusTVP}}:
Let $(t_\ast,x_\ast)\in [0,T)\times C([0,T],H)$,
$t_0\in (0,T]$, $x_0\in\mathcal{X}^F(t_\ast,x_\ast)$ with $(t_0,x_0)\in\mathrm{dom}\,v$,
and $\varphi\in\overline{\Phi}_-(t_0,x_0)$ with corresponding number $\eps>0$.
Noting that $x_0\in\mathcal{X}^F(t,x_0)$ for every $t\in [t_\ast,T]$, we can deduce with 
 \eqref{E:DPP2} that, for every $\delta\in (0,\eps]$ with $t_0-\delta\ge t_\ast$, we have
\begin{align*}
0\ge v(t_0-\delta,x_0)-v(t_0,x_0)\ge \varphi(t_0-\delta,x_0)-\varphi(t_0,x_0).
\end{align*}
Thus
\begin{align*}
0&\le\varlimsup_{\delta\downarrow 0}\frac{1}{\delta}\int_{t_0-\delta}^{t_0} \partial_t\varphi(t,x_0)-\langle A(t,x_0(t))\,\partial_x\varphi(t,x_0)\rangle
+(f^{x_0}(t),\partial_x\varphi(t,x_0))\,dt\\
&\le \partial_t \varphi(t_0,x_0)+\varlimsup_{\delta\downarrow 0}\frac{1}{\delta}
\int_{t_0-\delta}^{t_0} \langle -A(t,x_0(t)),\partial_x\varphi(t,x_0)\rangle+(f^{x_0}(t),\partial_x\varphi(t,x_0))\,dt.
\end{align*}

(iv) Uniqueness follows from Theorem~\ref{T:u:Visc:comparison}.
This concludes the proof.
\end{proof}

\subsection{{\color{black} An example}}
{\color{black} We  consider a distributed control  problem for the heat equation
with l.s.c.~terminal cost. To this end, we fix a bounded domain $G$ in $\R^n$
with smooth boundary $\partial G$,  a constant $C_P>0$, and the set 
\begin{align*}
P:=\{z\in L^2(G):\, \norm{z}_{L^2(G)}\le C_P\},
\end{align*}
which will be used as control set.

Given $(t_0,x_0)\in [0,T)\times C([0,T],L^2(G))$ and 
a Borel-measurable function $a:[0,T]\times\R^n\to\R$ with
 $\int_G \abs{a(t,\xi)}^2\,d\xi\le (C_P)^2$ for all $t\in [0,T]$,
consider the Cauchy--Dirichlet problem 
 \begin{equation}\label{E:ContrHeat2}
\begin{split}
x_t(t,\xi)- \Delta_\xi x (t,\xi)&=a(t,\xi),\quad (t,\xi)\in (t_0,T)\times G,\\
x(t,\xi)&=0,\quad (t,\xi) \in (t_0,T]\times \partial G,\\
x(t,\xi)&=x_0(t), \quad  (t, \xi)\in   [0,t_0]\times G.
\end{split}
\end{equation}
As in Example~1.3 in \cite{BK18JFA}
(see also Chapter~23 in \cite{ZeidlerIIA} for a detailed treatment)  we can formulate \eqref{E:ContrHeat2} as
an  abstract evolution equation on $(V,H,V^\ast)$ of the form
\begin{align}\label{E:Abstract2}
x^\prime(t)+Ax(t)&=\mathbf{a}(t)\text{ a.e.~on $(t_0,T)$ with $x=x_0$ on $[0,t_0]$,}
\end{align}
where $H=L^2(G)$,  $V=H^1_0(G)$, and $(t,y)\mapsto Ay$, 
$[0,T]\times V\to V^\ast$, satisfies \textbf{H}($A$). 
Here, $A$ corresponds to $-\Delta_\xi$ and
 $\mathbf{a}:[0,T]\to  P$ 
 to the function $a=a(t,\xi)$ above. 
We write $x^{t_0,x_0,\mathbf{a}}$  for the solution of  \eqref{E:Abstract2} in $C([0,T],L^2(G))\cap W_{pq}(t_0,T)$.
Note that here  $p=2$. 
Existence as well as uniqueness of our solution  $x^{t_0,x_0,\mathbf{a}}$ are standard
results (see, e.g., \cite{ZeidlerIIA}).

Next, we specify the remaining data for our control problem.
As terminal cost, we use  $h:C([0,T],L^2(G))\to\R\cup\{+\infty\}$ defined by
\begin{align*}
h(x)=\begin{cases}
0&\text{if $x(t)\in K$ for all $t\in [0,T]$,}\\
+\infty &\text{otherwise,}
\end{cases},
\end{align*}
where 
\begin{align*}
K=\{z\in L^2(G):\,  \norm{z}_{L^2(G)}\le C_K\}
\end{align*}
 and  $C_K>$ is a constant. Note that $h$ is l.s.c.
As class of admissible controls, we use  the set of all Borel-measurable functions from $[0,T]$ to  $P$,
which we denote by $\mathcal{A}$. 

Given $(t_0,x_0)\in [0,T)\times C([0,T],L^2(G))$, our optimal control problem
is to find a  control 
$\tilde{\mathbf{a}}\in\mathcal{A}$
such that
\begin{align*}
h(x^{t_0,x_0,\tilde{\mathbf{a}}})=\inf \{h(x^{t_0,x_0,\mathbf{a}}):\, \mathbf{a}\in\mathcal{A}\}.
\end{align*}
The  corresponding  value function $v:[0,T]\times C([0,T],L^2(G))\to\R\cup\{+\infty\}$ is defined by
\begin{align*}
v(t_0,x_0):=\inf\{h(x^{t_0,x_0,\mathbf{a}}):\, \mathbf{a}\in\mathcal{A}\}=
\inf\{h(x):\, x\in\mathcal{X}^F(t_0,x_0)\},
\end{align*}
where $F(t,z)=P$ for all $(t,z)\in [0,T]\times L^2(G)$.

Note that $v$ is not continuous. In particular, $v(t_0,x_0)=+\infty$ whenever $x_0(s)\not\in K$ for some $s\in [0,t_0]$ and
$v(t_0,x_0)=0$ whenever $x_0(t)=0$ for all $t\in [0,t_0]$. Nevertheless, $v$ can be characterized as a unique 
nonsmooth solution of the corresponding Hamilton--Jacobi--Bellman equation.

Indeed, by Theorem~\ref{T:u:Visc:U},
$v$ is the unique l.s.c.~viscosity solution of
\begin{align*}
\begin{split}
&-\partial_t u+\langle Ax(t),\partial_x u\rangle -\inf_{\mathbf{p}\in P} (\mathbf{p},\partial_x u)=0\text{ in $[0,T)\times C([0,T],L^2(G))$},\\
&u(T,x)=h(x)\text{ on $C([0,T],L^2(G))$.}
\end{split}
\end{align*}}

\bibliographystyle{amsplain}
\bibliography{JK24}

\begin{bibdiv}
\begin{biblist}

\bib{AubinFrankowska_setValued}{book}{
      author={Aubin, Jean-Pierre},
      author={Frankowska, H\'{e}l\`ene},
       title={Set-valued analysis},
      series={Modern Birkh\"{a}user Classics},
   publisher={Birkh\"{a}user Boston, Inc., Boston, MA},
        date={2009},
        ISBN={978-0-8176-4847-3},
         url={https://doi.org/10.1007/978-0-8176-4848-0},
        note={Reprint of the 1990 edition},
}

\bib{AP89Yokohama}{article}{
      author={Avgerinos, Evgenios~P.},
      author={Papageorgiou, Nikolaos~S.},
       title={Viable solutions for evolution inclusions},
        date={1989},
        ISSN={0044-0523},
     journal={Yokohama Math. J.},
      volume={37},
      number={2},
       pages={101\ndash 108},
}

\bib{BardiCapuzzoDolcetta}{book}{
      author={Bardi, Martino},
      author={Capuzzo-Dolcetta, Italo},
       title={Optimal control and viscosity solutions of
  {H}amilton-{J}acobi-{B}ellman equations},
      series={Systems \& Control: Foundations \& Applications},
   publisher={Birkh\"auser Boston, Inc., Boston, MA},
        date={1997},
        ISBN={0-8176-3640-4},
         url={http://dx.doi.org.proxy.lib.umich.edu/10.1007/978-0-8176-4755-1},
        note={With appendices by M.~Falcone and P.~Soravia},
}

\bib{BK18JFA}{article}{
      author={Bayraktar, Erhan},
      author={Keller, Christian},
       title={Path-dependent {H}amilton-{J}acobi equations in infinite
  dimensions},
        date={2018},
        ISSN={0022-1236},
     journal={J. Funct. Anal.},
      volume={275},
      number={8},
       pages={2096\ndash 2161},
         url={https://doi.org/10.1016/j.jfa.2018.07.010},
}

\bib{BK22SICON}{article}{
      author={Bayraktar, Erhan},
      author={Keller, Christian},
       title={Path-dependent {H}amilton-{J}acobi equations with super-quadratic
  growth in the gradient and the vanishing viscosity method},
        date={2022},
        ISSN={0363-0129},
     journal={SIAM J. Control Optim.},
      volume={60},
      number={3},
       pages={1690\ndash 1711},
         url={https://doi.org/10.1137/21M1395557},
}

\bib{Carja12SICON}{article}{
      author={Carja, Ovidiu},
       title={The minimum time function for semilinear evolutions},
        date={2012},
        ISSN={0363-0129},
     journal={SIAM J. Control Optim.},
      volume={50},
      number={3},
       pages={1265\ndash 1282},
}

\bib{Carja07book}{book}{
      author={C\^{a}rj\u{a}, Ovidiu},
      author={Necula, Mihai},
      author={Vrabie, Ioan~I.},
       title={Viability, invariance and applications},
      series={North-Holland Mathematics Studies},
   publisher={Elsevier Science B.V., Amsterdam},
        date={2007},
      volume={207},
        ISBN={978-0-444-52761-5},
}

\bib{Carja09TAMS}{article}{
      author={C\^{a}rj\u{a}, Ovidiu},
      author={Necula, Mihai},
      author={Vrabie, Ioan~I.},
       title={Necessary and sufficient conditions for viability for semilinear
  differential inclusions},
        date={2009},
        ISSN={0002-9947},
     journal={Trans. Amer. Math. Soc.},
      volume={361},
      number={1},
       pages={343\ndash 390},
         url={https://doi.org/10.1090/S0002-9947-08-04668-0},
}

\bib{Frankowska93SICON}{article}{
      author={Frankowska, H\'el\`ene},
       title={Lower semicontinuous solutions of {H}amilton-{J}acobi-{B}ellman
  equations},
        date={1993},
        ISSN={0363-0129},
     journal={SIAM J. Control Optim.},
      volume={31},
      number={1},
       pages={257\ndash 272},
         url={https://doi.org/10.1137/0331016},
      review={\MR{1200233}},
}

\bib{GP23JFA}{article}{
      author={Gomoyunov, M.~I.},
      author={Plaksin, A.~R.},
       title={Equivalence of minimax and viscosity solutions of path-dependent
  {H}amilton-{J}acobi equations},
        date={2023},
        ISSN={0022-1236,1096-0783},
     journal={J. Funct. Anal.},
      volume={285},
      number={11},
       pages={Paper No. 110155, 41 pp.},
         url={https://doi.org/10.1016/j.jfa.2023.110155},
}

\bib{GLP21AMO}{article}{
      author={Gomoyunov, Mikhail~I.},
      author={Lukoyanov, Nikolai~Yu.},
      author={Plaksin, Anton~R.},
       title={Path-dependent {H}amilton-{J}acobi equations: {T}he minimax
  solutions revised},
        date={2021},
        ISSN={0095-4616,1432-0606},
     journal={Appl. Math. Optim.},
      volume={84},
       pages={S1087\ndash S1117},
         url={https://doi.org/10.1007/s00245-021-09794-4},
}

\bib{Goreac11}{article}{
      author={Goreac, Dan},
       title={Viability of stochastic semi-linear control systems via the
  quasi-tangency condition},
        date={2011},
        ISSN={0265-0754},
     journal={IMA J. Math. Control Inform.},
      volume={28},
      number={3},
       pages={391\ndash 415},
         url={https://doi.org/10.1093/imamci/dnr003},
}

\bib{HP2}{book}{
      author={Hu, Shouchuan},
      author={Papageorgiou, Nikolas~S.},
       title={Handbook of multivalued analysis. {V}ol. {II}: {A}pplications},
      series={Mathematics and its Applications},
   publisher={Kluwer Academic Publishers, Dordrecht},
        date={2000},
      volume={500},
        ISBN={0-7923-6164-4},
         url={https://doi.org/10.1007/978-1-4615-4665-8_17},
}

\bib{Keller24}{article}{
      author={Keller, Christian},
       title={Mean viability theorems and second-order {H}amilton-{J}acobi
  equations},
        date={2024},
        ISSN={0363-0129,1095-7138},
     journal={SIAM J. Control Optim.},
      volume={62},
      number={3},
       pages={1615\ndash 1642},
         url={https://doi.org/10.1137/23M1550438},
}

\bib{Liu11Nonlin}{article}{
      author={Liu, Wei},
       title={Existence and uniqueness of solutions to nonlinear evolution
  equations with locally monotone operators},
        date={2011},
        ISSN={0362-546X,1873-5215},
     journal={Nonlinear Anal.},
      volume={74},
      number={18},
       pages={7543\ndash 7561},
         url={https://doi.org/10.1016/j.na.2011.08.018},
}

\bib{LiuRoeckner10JFA}{article}{
      author={Liu, Wei},
      author={R\"{o}ckner, Michael},
       title={S{PDE} in {H}ilbert space with locally monotone coefficients},
        date={2010},
        ISSN={0022-1236,1096-0783},
     journal={J. Funct. Anal.},
      volume={259},
      number={11},
       pages={2902\ndash 2922},
         url={https://doi.org/10.1016/j.jfa.2010.05.012},
}

\bib{Lukoyanov03a}{article}{
      author={Lukoyanov, N.~Yu.},
       title={Functional {H}amilton-{J}acobi type equations in ci-derivatives
  for systems with distributed delays},
        date={2003},
        ISSN={1229-1595},
     journal={Nonlinear Funct. Anal. Appl.},
      volume={8},
      number={3},
       pages={365\ndash 397},
}

\bib{Lukoyanov06_DiffIneq}{article}{
      author={Lukoyanov, N.~Yu.},
       title={Differential inequalities for a nonsmooth value functional in
  control systems with an aftereffect},
        date={2006},
        ISSN={0081-5438,1531-8605},
     journal={Proc. Steklov Inst. Math.},
       pages={S103\ndash S114},
         url={https://doi.org/10.1134/s0081543806060095},
}

\bib{P92Yokohama}{article}{
      author={Papageorgiou, Nikolaos~S.},
       title={A viability result for nonlinear time dependent evolution
  inclusions},
        date={1992},
        ISSN={0044-0523},
     journal={Yokohama Math. J.},
      volume={40},
      number={1},
       pages={73\ndash 86},
}

\bib{Shi88Nonlin}{article}{
      author={Shi, Shu~Zhong},
       title={Nagumo type condition for partial differential inclusions},
        date={1988},
        ISSN={0362-546X,1873-5215},
     journal={Nonlinear Anal.},
      volume={12},
      number={9},
       pages={951\ndash 967},
         url={https://doi.org/10.1016/0362-546X(88)90077-6},
}

\bib{ZeidlerIIA}{book}{
      author={Zeidler, Eberhard},
       title={Nonlinear functional analysis and its applications. {II}/{A}},
   publisher={Springer-Verlag, New York},
        date={1990},
        ISBN={0-387-96802-4},
         url={http://dx.doi.org/10.1007/978-1-4612-0985-0},
        note={Linear monotone operators, Translated from the German by the
  author and Leo F. Boron},
}

\bib{ZeidlerIIB}{book}{
      author={Zeidler, Eberhard},
       title={Nonlinear functional analysis and its applications. {II}/{B}},
   publisher={Springer-Verlag, New York},
        date={1990},
        ISBN={0-387-97167-X},
         url={http://dx.doi.org.proxy.lib.umich.edu/10.1007/978-1-4612-0985-0},
}

\end{biblist}
\end{bibdiv}
\end{document}